\documentclass{interact}

\usepackage{mathtools}
\usepackage{enumitem}
\usepackage{booktabs,tabularx,array}
\usepackage{ragged2e}
\usepackage{tikz}
\usetikzlibrary{calc}
\usepackage{microtype}
\usepackage[numbers,sort&compress,merge]{natbib}
\bibpunct[, ]{(}{)}{,}{n}{,}{,}

\usepackage[hidelinks]{hyperref}
\hypersetup{
  pdftitle={Finite-mass soliton-type rigidity and four-channel reduction for the three-dimensional nonradial focusing energy-critical nonlinear Schrodinger equation},
  pdfauthor={Pang-Hung Chung and Dan Han},
  pdfsubject={Focusing energy-critical nonlinear Schrodinger equation},
  pdfkeywords={energy-critical; nonlinear Schrodinger equation; focusing equation; below-threshold scattering; concentration compactness}
}
\numberwithin{equation}{section}
\allowdisplaybreaks

\theoremstyle{plain}
\newtheorem{theorem}{Theorem}[section]
\newtheorem{proposition}[theorem]{Proposition}
\newtheorem{lemma}[theorem]{Lemma}
\newtheorem{corollary}[theorem]{Corollary}
\theoremstyle{definition}
\newtheorem{definition}[theorem]{Definition}
\theoremstyle{remark}
\newtheorem{remark}[theorem]{Remark}

\newcommand{\R}{\mathbb R}
\newcommand{\C}{\mathbb C}
\newcommand{\ii}{\mathrm i}
\newcommand{\Rea}{\operatorname{Re}}
\newcommand{\Ima}{\operatorname{Im}}
\newcommand{\supp}{\operatorname{supp}}
\newcommand{\M}{\mathsf M}
\newcommand{\cS}{C_{\mathrm S}}
\newcommand{\F}{\mathcal F}
\newcommand{\Sspace}{\mathcal S}
\newcommand{\dd}{\,\mathrm d}
\newcommand{\one}{\mathbf 1}

\articletype{ORIGINAL ARTICLE}
\title{Finite-mass soliton-type rigidity and four-channel reduction for the three-dimensional nonradial focusing energy-critical nonlinear Schr\"odinger equation}
\author{%
\name{Pang-Hung Chung\textsuperscript{a} and Dan Han\textsuperscript{b}\thanks{CONTACT Dan Han. Email: dan.han@louisville.edu}}
\affil{\textsuperscript{a}Department of Applied Mathematics, Guangdong University of Education, Guangzhou 510640, P. R. China; \textsuperscript{b}Department of Mathematics, University of Louisville, Louisville, KY 40245, USA}%
}

\begin{document}
\maketitle

\begin{abstract}
The concentration--compactness channels below the ground-state threshold are investigated for the three-dimensional nonradial focusing energy-critical nonlinear Schr\"odinger equation. After one-sided normalization, a minimal critical element falls into four classes: the finite-time, rapid-cascade, bounded-scale finite-mass, and residual quasi-soliton channels. The first three classes are rigorously excluded. The main result shows, without radial symmetry, zero momentum, or a slow spatial center, that every finite-mass bounded-scale almost-periodic solution is identically zero. Consequently, any minimal counterexample to below-threshold scattering must lie in the residual quasi-soliton channel; if its scale is bounded, then it has infinite mass at every time.
\end{abstract}


\begin{amscode}
35Q55; 35B40; 35B44; 35P25
\end{amscode}

\section{Introduction}

We study the three-dimensional focusing energy-critical nonlinear Schr\"odinger equation
\begin{equation}\label{eq:nls}
 \ii\partial_tu+\Delta u+|u|^4u=0,
 \qquad (t,x)\in I\times\R^3.
\end{equation}
Equation~\eqref{eq:nls} is a Hamiltonian dispersive evolution in which the linear Schr\"odinger flow spreads spatial concentration, whereas the focusing quintic interaction can reinforce it. The competition between these two mechanisms is measured at the energy scale and permits both scattering-type dynamics and concentration scenarios. In three dimensions the quintic power is invariant at the homogeneous energy regularity, so the equation lies at the borderline where scale concentration cannot be ruled out by subcritical compactness. The purpose of the paper is to isolate and exclude several possible critical concentration channels below the ground-state threshold, with particular attention to the nonradial motion of the concentration core.

Here $u:I\times\R^3\to\C$, and $I\subset\R$ is a lifespan interval containing the initial time. Set $u_0:=u(0)$. For $1\le p\le\infty$, write $\|f\|_p:=\|f\|_{L^p(\R^3)}$, and define
\begin{align*}
 \dot H^1(\R^3)
 &:=\{f\in L^6(\R^3):\nabla f\in L^2(\R^3)\},
 &\|f\|_{\dot H^1}&:=\|\nabla f\|_2,\\
 H^1(\R^3)&:=L^2(\R^3)\cap\dot H^1(\R^3).
\end{align*}
Here $C_c^\infty(\R^3)$ denotes the space of smooth compactly supported functions. For a time interval $J$ and $1\le q,r<\infty$, define
\[
 \|v\|_{L_t^qL_x^r(J\times\R^3)}
 :=\left(\int_J\left(\int_{\R^3}|v(t,x)|^r\dd x\right)^{q/r}\dd t\right)^{1/q},
\]
with the corresponding essential-supremum definition when $q=\infty$ or $r=\infty$. In particular,
$\|v\|_{L^{10}_{t,x}(J\times\R^3)}:=\bigl(\int_J\int_{\R^3}|v(t,x)|^{10}\dd x\dd t\bigr)^{1/10}$.
The notation $L^{10}_{t,x,\mathrm{loc}}(J\times\R^3)$ means that this norm is finite on every compact subinterval of $J$. The notation $A\lesssim B$ means $A\le CB$ for a constant $C$ independent of the displayed variables. Equation \eqref{eq:nls} is invariant under the scaling
\[
 u(t,x)\longmapsto u_\lambda(t,x)
 :=\lambda^{1/2}u(\lambda^2t,\lambda x),
 \qquad \lambda>0,
\]
which preserves the homogeneous energy norm. Thus the quintic nonlinearity is exactly energy-critical in dimension three. The term ``critical'' is first of all a scaling statement: with $e^{\ii t\Delta}$ denoting the free Schr\"odinger group, the linear Strichartz estimate and Sobolev embedding yield
\[
 \|e^{\ii t\Delta}u_0\|_{L^{10}_{t,x}(\R\times\R^3)}
 \lesssim \|u_0\|_{\dot H^1}.
\]
When the right-hand side is sufficiently small, the Duhamel map is a contraction in the critical spacetime norm, giving global small-data solutions and scattering; an early systematic energy-space framework is due to Ginibre and Velo \cite{GinibreVelo1985}. For large data, however, the scaling $u\mapsto u_\lambda$ can compress the spatial scale to $\lambda^{-1}$ and the time scale to $\lambda^{-2}$ while leaving the energy and the $\dot H^1$ size unchanged. Translation symmetry also permits the concentration core to move through space. Conservation of energy alone therefore does not provide the compactness available in subcritical problems: loss of control can occur only through concentration of scale, drift of the center, or a coupling of the two. These are precisely the phenomena addressed by concentration compactness and by the scale--translation invariant estimates developed below.

For $f\in\dot H^1(\R^3)$, define the energy by
\begin{equation}\label{eq:energy}
 E(f):=\frac12\int_{\R^3}|\nabla f(x)|^2\dd x
       -\frac16\int_{\R^3}|f(x)|^6\dd x.
\end{equation}
Let
\[
 W(x):=\left(1+\frac{|x|^2}{3}\right)^{-1/2}.
\]
Then $W$ is the positive Aubin--Talenti ground state, satisfies $-\Delta W=W^5$, and is an optimizer for the sharp Sobolev inequality \cite{Aubin1976,Talenti1976,Lieb1983}. For this initial datum, we work throughout on the stable below-threshold branch
\begin{equation}\label{eq:threshold}
 E(u_0)<E(W),
 \qquad
 \|\nabla u_0\|_2<\|\nabla W\|_2.
\end{equation}
Besides placing the data below the ground state, this condition yields a uniform coercive gap through the sharp Sobolev inequality: along the entire lifespan there exists $c_0>0$ such that
\[
 \|\nabla u(t)\|_2^2-\|u(t)\|_6^6
 \ge c_0\|\nabla u(t)\|_2^2.
\]
The negative potential energy is the main focusing difficulty, and this positive gap is the basic source from which Morawetz positivity will be recovered.

A strong solution on an interval $J\subset\R$ means a function
\[
 u\in C(J;\dot H^1(\R^3))\cap L^{10}_{t,x,\mathrm{loc}}(J\times\R^3)
\]
that satisfies, for every $t,t_0\in J$,
\[
 u(t)=e^{\ii(t-t_0)\Delta}u(t_0)
 +\ii\int_{t_0}^t e^{\ii(t-s)\Delta}|u(s)|^4u(s)\dd s.
\]
It is maximal-lifespan if it has no extension as a strong solution beyond either endpoint of $J$. If $\sup J=+\infty$, it scatters forward when there exists $u_+\in\dot H^1(\R^3)$ such that
\[
 \lim_{t\to+\infty}\|u(t)-e^{\ii t\Delta}u_+\|_{\dot H^1}=0;
\]
backward scattering is defined analogously. The below-threshold scattering problem asks whether every datum satisfying \eqref{eq:threshold} generates a global strong solution that scatters in both time directions.

To describe critical concentration that could obstruct scattering, let $N:I\to(0,\infty)$ and $x:I\to\R^3$ denote a frequency scale and a spatial center. A solution is called almost periodic modulo scaling and translation if the normalized orbit
\[
 \mathcal K_u
 :=\left\{
 N(t)^{-1/2}u\!\left(t,x(t)+\frac{\cdot}{N(t)}\right):t\in I
 \right\}
\]
is precompact in $\dot H^1(\R^3)$. Here $N(t)^{-1}$ is the characteristic physical scale at time $t$, while $x(t)$ only locates the concentration core and is not assumed to be differentiable. The definition retains both critical symmetries and makes clear the central nonradial difficulty: compactness is available only in moving coordinates and at a varying scale, not in a fixed physical frame.

Whether every solution satisfying \eqref{eq:threshold} scatters globally is the central question in the three-dimensional nonradial focusing problem. It is essential to distinguish the focusing and defocusing signs. In dimension $d\ge3$, write the energy-critical family as
\begin{equation}\label{eq:sign-family}
 \ii\partial_tu+\Delta u+\mu |u|^{\frac4{d-2}}u=0,
 \qquad \mu\in\{-1,+1\},
\end{equation}
where $\mu=-1$ is defocusing and $\mu=+1$ is focusing. The corresponding energy is
\begin{equation}\label{eq:sign-energy}
 E_\mu(u)
 :=\frac12\|\nabla u\|_2^2
   -\mu\frac{d-2}{2d}
    \|u\|_{\frac{2d}{d-2}}^{\frac{2d}{d-2}}.
\end{equation}
For the family \eqref{eq:sign-family} with energy \eqref{eq:sign-energy}, $E_{-1}$ controls the critical gradient norm for all defocusing energy data, and there is no nonzero finite-energy stationary solution analogous to the focusing ground state. In the focusing case, the energy contains a negative potential term and
\[
 W_d(x):=\left(1+\frac{|x|^2}{d(d-2)}\right)^{-\frac{d-2}{2}}
\]
is itself a nonscattering stationary solution. Scattering can therefore be discussed only relative to the sharp threshold and its stable branch. The focusing theory is not obtained by merely changing a sign in the defocusing proof. The principal structural differences are summarized below.

\begin{table}[htbp]
\centering
\caption{Structural comparison of the defocusing and focusing energy-critical NLS.}
\label{tab:focus-defocus}
\small
\setlength{\tabcolsep}{4pt}
\renewcommand{\arraystretch}{1.18}
\begin{tabularx}{\textwidth}{>{\RaggedRight\arraybackslash}p{2.2cm}>{\RaggedRight\arraybackslash}X>{\RaggedRight\arraybackslash}X}
\toprule
Feature & Defocusing case $\mu=-1$ & Focusing case $\mu=+1$ \\
\midrule
Energy and coercivity
& $E_{-1}=\frac12\|\nabla u\|_2^2+\frac{d-2}{2d}\|u\|_{2d/(d-2)}^{2d/(d-2)}$ directly controls the $\dot H^1$ norm for all energy data.
& $E_{+1}=\frac12\|\nabla u\|_2^2-\frac{d-2}{2d}\|u\|_{2d/(d-2)}^{2d/(d-2)}$ is coercive only below the ground-state threshold and on the stable gradient branch.\\
Stationary objects and threshold
& The associated elliptic equation has no nonzero finite-energy solution, so there is no stationary ground-state obstruction.
& A ground state $W_d$ exists and satisfies $-\Delta W_d=W_d^{(d+2)/(d-2)}$; it is an optimizer, a threshold object, and an actual nonscattering solution.\\
Morawetz/virial sign
& The nonlinear term has the favorable sign, and interaction Morawetz estimates can directly produce positive spacetime control.
& The potential term has the unfavorable sign; positivity must be recovered through below-threshold coercivity, localization, frequency truncation, or a new covariance structure.\\
Global scattering
& Global well-posedness and scattering for general energy data are known in $d=3$, $d=4$, and $d\ge5$ \cite{CKSTT2008,RyckmanVisan2007,Visan2007}.
& The below-threshold theory depends on dimension and symmetry; the general theory is complete for $d\ge4$, whereas the three-dimensional nonradial endpoint has special infrared and center-drift obstructions.\\
Rigidity of minimal elements
& Frequency-localized interaction Morawetz estimates and long-time Strichartz estimates exploit global positivity \cite{CKSTT2008,KillipVisan2012}.
& Concentration compactness first produces a below-threshold minimal element; rigidity must additionally address ground-state concentration, infinite-mass infrared tails, and a moving center.\\
\bottomrule
\end{tabularx}
\end{table}

The contrast in Table~\ref{tab:focus-defocus} isolates the two additional focusing obstructions: the ground-state threshold and the unfavorable Morawetz sign. The radial defocusing theory in three dimensions was initiated by Bourgain \cite{Bourgain1999}, and Grillakis later treated large radial data by a different regularity and Morawetz argument \cite{Grillakis2000}. Tao extended radial global well-posedness and scattering to all $d\ge3$ \cite{Tao2005}; Tao and Visan developed the critical stability theory needed in higher dimensions, where the nonlinearity is no longer locally Lipschitz \cite{TaoVisan2005}. For general data, the three-dimensional energy-space theory was proved by Colliander, Keel, Staffilani, Takaoka, and Tao \cite{CKSTT2008}; the four-dimensional result is due to Ryckman and Visan \cite{RyckmanVisan2007}, and the case $d\ge5$ to Visan \cite{Visan2007}. Killip and Visan later gave another three-dimensional proof based on concentration compactness and long-time Strichartz estimates \cite{KillipVisan2012}. Thus the defocusing problem has a complete energy-space scattering theory, but its decisive positivity disappears under the focusing sign.

Recent defocusing work has increasingly concerned probabilistic data or models in which the homogeneous structure is broken. Shen, Soffer, and Wu established almost-sure scattering for nonradial energy-critical equations in dimensions three and four \cite{ShenSofferWu2021}, and Marsden developed a probabilistic scattering theory in dimensions $d>6$ \cite{Marsden2022}. For spatially inhomogeneous energy-critical equations, Park proved radial defocusing scattering in several dimensions and parameter ranges \cite{Park2023}, while Yang, Zhang, and Liu completed the general nonradial defocusing theory for all $d\ge3$ \cite{YangZhangLiu2026}. These results show that defocusing positivity can be combined with randomization, Lorentz-space estimates, and concentration compactness; they do not remove the ground state or the negative potential energy in the homogeneous focusing problem.

The dimensional dependence of the focusing theory is already visible in the infrared tail of the ground state. The preceding explicit formula gives
\[
 W_d(x)\sim |x|^{-(d-2)}\quad(|x|\to\infty),
\]
and therefore
\begin{equation}\label{eq:ground-state-L2-dimension}
 W_d\in L^2(\R^d)\quad\Longleftrightarrow\quad d\ge5.
\end{equation}
The criterion \eqref{eq:ground-state-L2-dimension} shows that negative regularity and double-Duhamel arguments can recover finite mass when $d\ge5$; dimension four is logarithmically divergent; and the infrared divergence is stronger in dimension three. The high-dimensional proof cannot simply be pushed down to $d=3$. Beyond below-threshold scattering, Duyckaerts and Merle classified radial threshold dynamics in dimensions three, four, and five \cite{DuyckaertsMerle2009}, and Li and Zhang extended the threshold analysis to $d\ge6$ \cite{LiZhang2009}. The following table records representative results most relevant to the present setting and is not intended as an exhaustive account of threshold dynamics.

\begin{table}[htbp]
\centering
\caption{Representative below-threshold scattering results for the focusing energy-critical NLS (non-exhaustive).}
\label{tab:focusing-status}
\scriptsize
\setlength{\tabcolsep}{2.2pt}
\renewcommand{\arraystretch}{1.15}
\begin{tabularx}{\textwidth}{>{\Centering\arraybackslash}p{1.30cm}>{\RaggedRight\arraybackslash}p{2.15cm}>{\RaggedRight\arraybackslash}p{1.75cm}>{\RaggedRight\arraybackslash}X>{\RaggedRight\arraybackslash}p{2.35cm}}
\toprule
Dimension & Critical nonlinearity & Data/symmetry & Below-threshold result and status & Representative method/reference \\
\midrule
$d=3$
& Quintic $|u|^4u$
& Radial $\dot H^1$
& The global well-posedness/scattering versus blow-up dichotomy below the ground state is known.
& Concentration--compactness rigidity and localized virial; Kenig--Merle \cite{KenigMerle2006}.\\
$d=3$
& Quintic $|u|^4u$
& General nonradial $\dot H^1$
& Unconditional below-threshold scattering does not follow from the defocusing or higher-dimensional theory. A general low-speed criterion and a cylindrical conditional criterion are available \cite{ChungHan2026,ChungHanCylindrical2026}; the present paper removes drift and symmetry assumptions in the finite-mass bounded-scale channel and also closes the finite-time and rapid-cascade channels.
& Since $W\notin L^2$, infrared tails and center drift coexist; a local Galilean/interaction-Morawetz rigidity mechanism is required.\\
$d=4$
& Cubic $|u|^2u$
& General nonradial $\dot H^1$
& Global well-posedness and scattering below the ground state are known.
& Long-time Strichartz estimates overcome the logarithmic $L^2$ endpoint; Dodson \cite{Dodson2019}.\\
$d\ge5$
& $|u|^{4/(d-2)}u$
& General nonradial $\dot H^1$
& Global well-posedness and scattering below the ground state are known.
& Negative regularity, double Duhamel, and recovery of finite mass; Killip--Visan \cite{KillipVisan2010}.\\
\bottomrule
\end{tabularx}
\end{table}

Table~\ref{tab:focusing-status} emphasizes that the rigidity mechanisms at the three-dimensional nonradial endpoint differ from those in dimensions four and higher. Recent progress directly concerning the homogeneous equation includes the global weak solutions and weak--strong uniqueness constructed by Cheng, Guo, and Zheng through energy-critical Ginzburg--Landau approximation \cite{ChengGuoZheng2023}, and the extension by Ma, Miao, Murphy, and Zheng of four-dimensional nonradial threshold dynamics at ground-state energy and below-ground-state gradient \cite{MaMiaoMurphyZheng2025}. These concern, respectively, weak-solution construction and threshold classification, and do not replace the strong-solution rigidity problem studied here.

A substantial part of recent focusing energy-critical research concerns potentials or spatially inhomogeneous coefficients. Yang, Zeng, and Zhang classified threshold solutions in the presence of an inverse-square potential \cite{YangZengZhang2020}. Guzm\'an and Murphy treated the three-dimensional nonradial energy-critical inhomogeneous equation \cite{GuzmanMurphy2021}; Guzm\'an and Xu \cite{GuzmanXu2024} and Park \cite{Park2024} subsequently extended nonradial below-threshold scattering to broader dimensional and parameter ranges. Liu, Yang, and Zhang classified threshold dynamics for inhomogeneous models in dimensions three through five \cite{LiuYangZhang2024}; Campos, Farah, and Murphy constructed and classified special threshold solutions for a three-dimensional energy-critical inhomogeneous cubic model \cite{CamposFarahMurphy2026}; and Ma, Song, Yang, and Zhang recently proved threshold scattering for a model with a repulsive inverse-square potential \cite{MaSongYangZhang2026}. Since spatial coefficients break translation symmetry, their concentration cores cannot drift freely as in the homogeneous three-dimensional equation. These models therefore provide a useful contrast rather than a direct substitute for the present result.

For systems and competing nonlinearities, Farah and Hespanha proved a scattering--blow-up dichotomy for radial focusing energy-critical Schr\"odinger systems \cite{FarahHespanha2025}, while Bellazzini, Dinh, and Forcella gave an interaction-Morawetz proof of below-ground-state scattering for a three-dimensional nonradial equation with competing nonlinearities \cite{BellazziniDinhForcella2022}. The latter and the present work both use functionals in the difference variable, but the pure focusing quintic term considered here requires positivity to be rebuilt from local Galilean covariance and below-threshold well coercivity.

Kenig and Merle introduced the critical-element method in the three-dimensional radial setting \cite{KenigMerle2006}: if scattering fails, linear profile decomposition, nonlinear stability, and Palais--Smale compactness yield a nonscattering solution of minimal energy whose orbit is precompact modulo symmetries. The variational foundation is Lions's concentration--compactness principle \cite{Lions1984I,Lions1984II}; G\'erard precisely described the loss of compactness caused by translation and scaling in the critical Sobolev embedding \cite{Gerard1998}. Energy-critical profile decomposition and critical regularity theory may also be found in \cite{BahouriGerard1999,Keraani2001,KillipVisan2013}. In dimensions $d\ge5$, Killip and Visan excluded below-threshold focusing minimal elements using negative regularity \cite{KillipVisan2010}; in $d=4$, Dodson handled the logarithmic endpoint through long-time Strichartz estimates \cite{Dodson2019}. The three-dimensional nonradial quintic problem lies at a lower infrared endpoint: some frequency recursions are sign-independent, but excluding a long-time compact orbit must simultaneously overcome negative potential energy, infinite mass, and center drift.

These results show that concentration compactness, modulation, virial/Morawetz identities, and below-threshold variational structure remain common tools across focusing critical problems. What varies with the model is translation symmetry, infrared mass, and the sign of the nonlinear term. Our purpose is not to subsume the neighboring models, but to isolate and exclude the finite-mass bounded-scale obstruction in the homogeneous three-dimensional nonradial equation.

The nonradial pure-energy setting presents two related obstacles. First, data in $\dot H^1(\R^3)$ need not lie in $L^2(\R^3)$, so total mass, center of mass, and the usual finite-mass virial functional may be undefined. Second, almost periodicity gives no differentiability, velocity bound, or sublinear drift for $x(t)$. A localized virial identity around a preselected center produces center-velocity and shell-flux terms, thereby reintroducing first-moment or drift assumptions that are not available. The classical interaction Morawetz identity and its tensor-product formulation go back to \cite{Morawetz1968,LinStrauss1978,CGT2009}; the bilinear virial identities of Planchon and Vega further reveal the geometry of difference-variable functionals \cite{PlanchonVega2009}. In the focusing case, however, tensor positivity alone is defeated by the negative potential energy.

Two recent companion papers by the authors address neighboring rigidity mechanisms for the same three-dimensional focusing problem. In \cite{ChungHan2026}, a fixed-center localized virial identity excludes bounded-scale critical elements under an $L_t^\infty L_x^q$ bound, $2\le q\le6$, together with a corresponding sequential low-speed condition on the concentration center. In the cylindrically symmetric class, \cite{ChungHanCylindrical2026} isolates a best fixed-axis window condition and a low-frequency tail smallness condition, and excludes these entrances by a shifted localized virial argument and, at the finite-mass level, an axial zero-momentum normalization. The present paper has a different scope: in the finite-mass bounded-scale channel it uses a self-interaction Morawetz functional depending only on the difference variable $x-y$, so no spatial center is selected or tracked, and neither symmetry, center-speed control, nor zero total momentum is assumed.

The core mechanism is summarized by a local covariance block. Fix $R>0$, $c\in\R^3$, and a radial cutoff $\zeta\in C_c^\infty(\R^3)$. Set $B(c,R):=\{x\in\R^3:|x-c|<R\}$, $\zeta_{R,c}(x):=\zeta((x-c)/R)$, and $p[u](x):=\Ima(\overline{u(x)}\nabla u(x))$. At a fixed time, with $u=u(t,\cdot)$, define
\[
 M_{R,c}:=\int\zeta_{R,c}^2|u|^2\dd x,\qquad
 P_{R,c}:=\int\zeta_{R,c}^2p[u]\dd x,\qquad
 K_{R,c}:=\int\zeta_{R,c}^2(|\nabla u|^2-|u|^6)\dd x.
\]
When $M_{R,c}>0$, choose the optimal local Galilean parameter $\xi_{R,c}:=P_{R,c}/M_{R,c}$. Then
\[
 M_{R,c}K_{R,c}-|P_{R,c}|^2
 =M_{R,c}\int_{\R^3}\zeta_{R,c}^2
 \left(\left|\nabla\bigl(e^{-\ii x\cdot\xi_{R,c}}u\bigr)\right|^2-|u|^6\right)\dd x.
\]
Thus the squared-momentum term is not an error to be estimated separately: it subtracts the average transport velocity from the local kinetic energy, leaving a positive covariance. Below-threshold coercivity turns this covariance into positive localized energy, at the cost of only an $O(R^{-2})$ cutoff error.

To avoid losing parameters at a preselected physical radius, fix $R_0>0$ and $J\ge1$, and average logarithmically over $R\in[R_0,e^JR_0]$ with measure $\dd R/R$. Differentiating the resulting self-interaction functional produces a flat main block, equal to the logarithmic average of the covariance above, together with explicit kernel-derivative and bi-Laplacian remainders. Bounded-scale almost periodicity gives a time-uniform positive lower bound for the flat block, logarithmic averaging reduces the remainders to $O(J^{-1})+O((JR_0^2)^{-1})$, and finite mass bounds the functional at the time endpoints. Fixing first $R_0$ and then a sufficiently large $J$ yields the contradiction. The construction contains neither $x'(t)$ nor $|x(t)|=o(t)$ and directly addresses center drift; the exact functionals are defined before use in Sections~5--6.

We now state the four-channel decomposition of a minimal critical element. If below-threshold scattering fails, concentration compactness permits one to reselect a nonzero one-sided normalized minimal critical element. Let $T_+\in(0,\infty]$ be its forward maximal-lifespan endpoint and let $N:[0,T_+)\to(0,\infty)$ be its frequency scale; thus
\[
 u_c:[0,T_+)\times\R^3\to\C.
\] The solution does not scatter forward, and in the global case it may be normalized so that $N(t)\ge1$. Minimality also implies the no-waste Duhamel formula; it is therefore not an additional assumption in the following classification. When $T_+=\infty$, set
\[
 \mathsf K_+(u_c):=\int_0^\infty N(t)^{-1}\dd t.
\]
Let $\mathfrak T(u_c)$ denote one of the following four mutually exclusive channels:
\begin{enumerate}[label=\textup{(\Roman*)},leftmargin=2.8em]
 \item \textbf{$\mathsf{FT}$: finite-time channel.} $T_+<\infty$.

 \item \textbf{$\mathsf{RC}$: rapid low-to-high frequency cascade.} $T_+=\infty$ and $\mathsf K_+(u_c)<\infty$.

 \item \textbf{$\mathsf{FM}$: finite-mass bounded-scale channel.}
 \begin{equation}\label{eq:intro-channel-FM}
  T_+=\infty,
  \qquad
  \sup_{t\ge0}N(t)<\infty,
  \qquad
  u_c(t_0)\in L^2(\R^3)\ \text{for some }t_0\ge0.
 \end{equation}
 Since $N(t)\ge1$ and the time half-line is infinite, \eqref{eq:intro-channel-FM} automatically implies $\mathsf K_+(u_c)=\infty$.

 \item \textbf{$\mathsf{RQ}$: residual quasi-soliton channel.}
 \begin{equation}\label{eq:intro-channel-RQ}
  \left\{
  \begin{aligned}
   &T_+=\infty,
   \qquad
   \mathsf K_+(u_c)=\infty,\\
   &\sup_{t\ge0}N(t)=\infty
   \quad\text{or}\quad
   u_c(t)\notin L^2(\R^3)\ \text{for every }t\ge0.
  \end{aligned}
  \right.
 \end{equation}
\end{enumerate}
The classification first separates finite and infinite lifespan, then divides the global branch according to whether $\mathsf K_+$ is finite, and finally removes the bounded-scale finite-mass subclass $\mathsf{FM}$ from the branch with infinite $\mathsf K_+$. Its complement is exactly $\mathsf{RQ}$ as defined in \eqref{eq:intro-channel-RQ}. We exclude $\mathsf{FT}$, $\mathsf{RC}$, and $\mathsf{FM}$, leaving $\mathsf{RQ}$ as the residual channel for the full below-threshold scattering problem.

Among the four classes, $\mathsf{FM}$ most directly displays the new obstruction. Finite mass makes the self-interaction Morawetz functional meaningful, and bounded scale prevents the concentration core from disappearing at every fixed physical radius, but the spatial center may still drift arbitrarily. The local Galilean covariance and logarithmic scale average above are designed precisely for this situation.

\begin{theorem}[Bounded-scale finite-mass rigidity]\label{thm:finite-mass-rigidity}
Let $u$ be a global strong solution to \eqref{eq:nls} on $[0,\infty)$ such that
\[
 u\in C([0,\infty);H^1(\R^3)).
\]
Assume further that
\begin{equation}\label{eq:main-threshold}
 E(u(0))<E(W),
 \qquad
 \|\nabla u(0)\|_2<\|\nabla W\|_2,
\end{equation}
that $u$ is almost periodic modulo scaling and translation, and that there exist constants $0<N_-\le N_+<\infty$ such that its frequency scale $N:[0,\infty)\to(0,\infty)$ obeys
\begin{equation}\label{eq:main-bounded-scale}
 N_-\le N(t)\le N_+
 \qquad(t\ge0).
\end{equation}
Then
\[
 u\equiv0.
\]
\end{theorem}

In particular, there is no nonzero global finite-mass bounded-scale almost-periodic solution; such a solution would of course scatter forward. No radial symmetry, zero total momentum, differentiability of $x(t)$, sublinear drift $|x(t)|=o(t)$, or finite first spatial moment is assumed.

Theorem~\ref{thm:finite-mass-rigidity} applies to the $\mathsf{FM}$ channel. If a minimal critical element lies in $L^2$ at one time, persistence of regularity and conservation of mass propagate finite mass throughout its lifespan, while one-sided normalization provides the lower scale bound. Together with the upper scale bound in $\mathsf{FM}$, all assumptions of the theorem are satisfied. The channels $\mathsf{FT}$ and $\mathsf{RC}$ are excluded by finite-endpoint mass decay and long-time frequency recursion, respectively. The main rigidity theorem and these complementary mechanisms therefore close the first three channels.

\begin{theorem}[Exclusion of three channels]\label{thm:closed-channels}
On the stable below-threshold branch \eqref{eq:threshold}, there is no nonzero one-sided normalized minimal critical element satisfying
\[
 \mathfrak T(u_c)\in\{\mathsf{FT},\mathsf{RC},\mathsf{FM}\}.
\]
Equivalently, any normalized minimal counterexample that survives the concentration--compactness reduction must satisfy
\[
 \mathfrak T(u_c)=\mathsf{RQ}.
\]
\end{theorem}

Theorem~\ref{thm:closed-channels} is a rigorous rigidity reduction, not a prior exclusion of the residual channel. Concentration compactness also says that failure of below-threshold scattering produces a one-sided normalized minimal critical element, whereas any such nonzero element is itself an obstruction to scattering. Thus the full scattering problem is exactly equivalent to the nonexistence of $\mathsf{RQ}$.

\begin{corollary}[Residual-channel criterion]\label{cor:residual-channel}
The following are equivalent:
\begin{enumerate}[label=\textup{(\alph*)},leftmargin=2.4em]
 \item Every datum satisfying \eqref{eq:threshold} generates a unique global strong solution that scatters in both time directions.
 \item There is no nonzero one-sided normalized minimal critical element satisfying
 \[
  \mathfrak T(u_c)=\mathsf{RQ}.
 \]
\end{enumerate}
Consequently, full below-threshold scattering is equivalent to the nonexistence of the residual quasi-soliton channel $\mathsf{RQ}$.
\end{corollary}

\begin{remark}\label{rem:relation-low-speed}
The companion results \cite{ChungHan2026,ChungHanCylindrical2026} and Theorem~\ref{thm:finite-mass-rigidity} are complementary rather than nested. The fixed-center method of \cite{ChungHan2026} treats the general nonradial bounded-scale channel at several integrability levels $u\in L_t^\infty L_x^q$, $2\le q\le6$, but requires a $q$-dependent sequential low-speed condition; even at the finite-mass endpoint $q=2$, a sublinear drift condition remains. The cylindrical analysis in \cite{ChungHanCylindrical2026} exploits rotational symmetry to reduce the spatial motion to an axial center and closes two conditional entrances by shifted and finite-mass virial identities, the latter after axial zero-momentum normalization. By contrast, Theorem~\ref{thm:finite-mass-rigidity} is restricted to the finite-mass endpoint but removes all center-drift assumptions, cylindrical or radial symmetry, and zero-momentum normalization by using a difference-variable self-interaction Morawetz functional and subtracting the optimal local Galilean momentum. None of the three rigidity statements is used to prove either of the others.
\end{remark}

The remainder of the paper is organized as follows. Section~2 collects the local theory, concentration--compactness inputs, almost-periodic framework, coercivity estimates, and long-time Strichartz consequences used later. Section~3 excludes the finite-time channel by combining the no-waste representation with frequency-localized mass estimates. Section~4 rules out the rapid-cascade channel through scale divergence, negative regularity, and mass conservation. Section~5 constructs the logarithmically averaged kernels and records the local conservation laws underlying the interaction argument. Section~6 derives the exact self-interaction Morawetz identity and decomposes its time derivative into the principal coercive block and controllable errors. Section~7 removes the local momentum defect by the optimal ballwise Galilean transformation and proves the required local coercivity. Section~8 combines these ingredients with the local-mass budget and a low-regularity approximation to establish finite-mass rigidity. Section~9 assembles the preceding exclusions into the four-channel reduction and proves the residual-channel criterion. 

\section{Preliminaries and external results}

\subsection{Basic notation and notions of solution}

We use the standard notation for Lebesgue and Sobolev norms, with
$\|f\|_p:=\|f\|_{L^p(\R^3)}$ and $\|f\|_{\dot H^1}:=\|\nabla f\|_2$. For a Banach space $X$ and an interval $J$, set
\begin{align*}
 C_tX(J)&:=C(J;X),\\
 C_t\dot H_x^1(J\times\R^3)&:=C(J;\dot H^1(\R^3)),\\
 C_tH_x^1(J\times\R^3)&:=C(J;H^1(\R^3)).
\end{align*}
We write $B(x_0,R):=\{x\in\R^3:|x-x_0|<R\}$, $\one_E$ for the indicator of a set $E$, $C_c^\infty(\R^3)$ for the smooth compactly supported functions, and $\Sspace(\R^3)$ for the Schwartz class; a number $M\in2^{\mathbb Z}$ is called dyadic. For nonnegative quantities $A,B$, the notation $A\lesssim_\Lambda B$ means $A\le C_\Lambda B$, and $A\simeq_\Lambda B$ means both inequalities. The subscript $u$ indicates dependence only on the fixed solution and its below-threshold gap. We write $o_n(1)$ for a quantity tending to zero as $n\to\infty$, and $n\gg_\Lambda1$ when $n$ is sufficiently large with threshold depending only on $\Lambda$. For a subset $K$ of a Banach space $X$, set $\operatorname{dist}_X(f,K):=\inf_{g\in K}\|f-g\|_X$; the notation $K\Subset X$ means that the closure of $K$ is compact in $X$. Thus, for intervals, $J\Subset I$ means that $\overline J$ is a compact subset of $I$. Weak convergence, supports, the Kronecker delta, Einstein summation, and tensor products follow the standard conventions. A subscript on a differential operator indicates the variable in which it acts.

Our Fourier transform convention is
\[
 \widehat f(\xi):=(\F f)(\xi)
 :=\int_{\R^3}e^{-\ii x\cdot\xi}f(x)\dd x.
\]
Fix a radial function $\varphi\in C_c^\infty(\R^3)$ satisfying
\[
 0\le\varphi\le1,
 \qquad
 \varphi(\xi)=1\quad(|\xi|\le1),
 \qquad
 \varphi(\xi)=0\quad(|\xi|\ge2).
\]
For $M>0$, define the Littlewood--Paley operators by
\[
 \widehat{P_{\le M}f}(\xi):=\varphi(\xi/M)\widehat f(\xi),
 \qquad
 P_{>M}:=\operatorname{Id}-P_{\le M}.
\]
For dyadic $M$, set
\[
 P_M:=P_{\le M}-P_{\le M/2}.
\]
For a time interval $J\subset\R$ and $1\le q,r\le\infty$, use the mixed norm defined in the introduction. For $X\in\{\dot H^1(\R^3),H^1(\R^3)\}$, set
\[
 \|f\|_{L_t^\infty X(J)}:=\operatorname*{ess\,sup}_{t\in J}\|f(t)\|_X;
\]
we abbreviate these spaces by $L_t^\infty\dot H_x^1(J\times\R^3)$ and $L_t^\infty H_x^1(J\times\R^3)$, respectively. For a strong solution $u$ and a time interval $J$, define the critical scattering size
\[
 S_J(u):=\|u\|_{L_{t,x}^{10}(J\times\R^3)}.
\]

For $f\in L^2(\R^3)$, define the mass
\begin{equation}\label{eq:mass}
 \M(f):=\int_{\R^3}|f(x)|^2\dd x=\|f\|_2^2.
\end{equation}
If moreover $f\in H^1(\R^3)$, define the total momentum
\begin{equation}\label{eq:total-momentum}
 \mathcal P(f):=\Ima\int_{\R^3}\overline{f(x)}\,\nabla f(x)\dd x\in\R^3.
\end{equation}
The main proof does not assume $\mathcal P(u)=0$.

\begin{definition}[Strong solutions and scattering]\label{def:strong-scatter}
Let $I\subset\R$ be an interval. A function
\[
 u\in C_t\dot H_x^1(I\times\R^3)
 \cap L^{10}_{t,x,\mathrm{loc}}(I\times\R^3)
\]
is called a strong solution to \eqref{eq:nls} if, for all $t,t_0\in I$,
\[
 u(t)=e^{\ii(t-t_0)\Delta}u(t_0)
 +\ii\int_{t_0}^t e^{\ii(t-s)\Delta}|u(s)|^4u(s)\dd s.
\]
If $\sup I=+\infty$ and there exists $u_+\in\dot H^1$ such that
\[
 \lim_{t\to+\infty}
 \|u(t)-e^{\ii t\Delta}u_+\|_{\dot H^1}=0,
\]
then $u$ is said to scatter forward. Backward scattering is defined analogously.
\end{definition}
All later uses of ``strong solution'', ``forward scattering'', and ``backward scattering'' refer to Definition~\ref{def:strong-scatter}. If $u$ is an $H^1$ strong solution, then the mass \eqref{eq:mass}, momentum \eqref{eq:total-momentum}, and energy \eqref{eq:energy} are conserved. Implicit constants may change from line to line.

\subsection{Almost periodicity and critical channels}

\begin{definition}[Almost periodicity]\label{def:ap}
A maximal-lifespan strong solution $u:I\times\R^3\to\C$ is almost periodic modulo scaling and translation if there exist functions
\[
 N:I\to(0,\infty),
 \qquad
 x:I\to\R^3,
\]
such that the normalized orbit
\begin{equation}\label{eq:normalized-orbit}
 \mathcal K_u
 :=\left\{
 N(t)^{-1/2}u\left(t,x(t)+\frac{\cdot}{N(t)}\right):
 t\in I
 \right\}
\end{equation}
is precompact in $\dot H^1(\R^3)$. The functions $N(t)$ and $x(t)$ are called the frequency scale and spatial center, respectively.
\end{definition}

Definition~\ref{def:ap} is equivalent to the following spatial--frequency compactness property: for every $\eta>0$, there exists $C(\eta)>1$ such that, for every $t\in I$,
\begin{align}
 &\int_{|x-x(t)|\ge C(\eta)/N(t)}
 \bigl(|\nabla u(t,x)|^2+|u(t,x)|^6\bigr)\dd x
 \le\eta,
 \label{eq:ap-space}\\
 &\int_{|\xi|\le N(t)/C(\eta)}
 |\xi|^2|\widehat u(t,\xi)|^2\dd\xi
 +\int_{|\xi|\ge C(\eta)N(t)}
 |\xi|^2|\widehat u(t,\xi)|^2\dd\xi
 \le\eta.
 \label{eq:ap-frequency}
\end{align}
We use the normalized orbit \eqref{eq:normalized-orbit}, the spatial compactness estimate \eqref{eq:ap-space}, and the frequency compactness estimate \eqref{eq:ap-frequency}; no differentiability of $N(t)$ or $x(t)$ is required.

Set
\[
 \mathcal A
 :=\left\{f\in\dot H^1(\R^3):
 E(f)<E(W),\ \|\nabla f\|_2<\|\nabla W\|_2\right\}.
\]
For $0<E_*<E(W)$, let $u_f:I(f)\times\R^3\to\C$ denote the maximal-lifespan solution with datum $f\in\mathcal A$, and define the uniform scattering size
\begin{equation}\label{eq:uniform-scattering-size}
 \mathcal L(E_*)
 :=\sup\left\{
 S_{I(f)}(u_f):
 f\in\mathcal A,\ E(f)\le E_*
 \right\}\in[0,\infty].
\end{equation}
If below-threshold scattering fails, define the critical energy
\begin{equation}\label{eq:critical-energy}
 E_c
 :=\sup\left\{
 E_*\in(0,E(W)):\mathcal L(E_*)<\infty
 \right\}.
\end{equation}
Small-data scattering and the failure assumption give
\begin{equation}\label{eq:critical-energy-range}
 0<E_c<E(W),
 \qquad
 \mathcal L(E_*)<\infty\quad(0<E_*<E_c).
\end{equation}
A nonscattering solution of energy $E_c$ whose orbit is precompact modulo symmetries will be called a minimal critical element. For a maximal-lifespan strong solution $V:I(V)\times\R^3\to\C$, we abbreviate
\[
 S(V):=S_{I(V)}(V).
\]

For a global almost-periodic solution satisfying $N(t)\ge1$ and a time interval $J$, define
\[
 \mathsf K(J):=\int_JN(t)^{-1}\dd t,
 \qquad
 \mathsf K_+(u):=\mathsf K([0,\infty)).
\]
If $\mathsf K_+(u)<\infty$, we call the solution a rapid low-to-high frequency cascade; if $\mathsf K_+(u)=\infty$, we call it a quasi-soliton. This terminology concerns only the time--frequency integral and does not assume that $N(t)$ is bounded above. The four-channel decomposition in the introduction further divides the quasi-soliton branch into the bounded-scale finite-mass class $\mathsf{FM}$ and its complement $\mathsf{RQ}$.

\subsection{Below-threshold coercivity}

Define the well functional
\[
 K(f):=\|\nabla f\|_2^2-\|f\|_6^6,
 \qquad f\in\dot H^1(\R^3).
\]
The sharp Sobolev inequality is
\begin{equation}\label{eq:sharp-Sobolev}
 \|f\|_6\le\cS\|\nabla f\|_2,
 \qquad
 \cS^6=\|\nabla W\|_2^{-4}.
\end{equation}
The ground state satisfies the Pohozaev identities
\begin{equation}\label{eq:Pohozaev}
 \|W\|_6^6=\|\nabla W\|_2^2,
 \qquad
 E(W)=\frac13\|\nabla W\|_2^2.
\end{equation}
We first convert the sharp Sobolev inequality into the quantitative below-threshold coercivity used throughout the paper.

\begin{lemma}[Below-threshold coercivity]\label{lem:coercivity}
Fix $E_*<E(W)$. If $f\in\dot H^1(\R^3)$ satisfies
\[
 E(f)\le E_*,
 \qquad
 \|\nabla f\|_2<\|\nabla W\|_2,
\]
then there exist constants $\delta_0,c_0>0$, depending only on $E_*/E(W)$, such that
\begin{align}
 \|\nabla f\|_2^2
 &\le(1-\delta_0)\|\nabla W\|_2^2,
 \label{eq:gradient-gap}\\
 K(f)&\ge c_0\|\nabla f\|_2^2,
 \label{eq:K-positive}\\
 c_0\|\nabla f\|_2^2
 &\le E(f)\le\frac12\|\nabla f\|_2^2,
 \label{eq:E-equivalence}\\
 \cS^2\|f\|_6^4&\le1-\delta_0.
 \label{eq:L6-gap}
\end{align}
If $u$ is a strong solution satisfying \eqref{eq:threshold}, these estimates hold uniformly throughout its lifespan.
\end{lemma}

\begin{proof}
Set
\[
 y:=\frac{\|\nabla f\|_2^2}{\|\nabla W\|_2^2}\in[0,1).
\]
By \eqref{eq:sharp-Sobolev} and \eqref{eq:Pohozaev},
\begin{align}
 \|f\|_6^6
 &\le \cS^6\|\nabla f\|_2^6
 =\frac{\|\nabla f\|_2^6}{\|\nabla W\|_2^4}
 =y^3\|\nabla W\|_2^2,
 \label{eq:sob-y}\\
 E(f)
 &\ge \|\nabla W\|_2^2
 \left(\frac y2-\frac{y^3}{6}\right).
 \label{eq:energy-y}
\end{align}
The function
\[
 g(y):=\frac y2-\frac{y^3}{6}
\]
is strictly increasing on $[0,1]$, and \eqref{eq:Pohozaev} gives
\begin{equation}\label{eq:g-one}
 g(1)=\frac13=\frac{E(W)}{\|\nabla W\|_2^2}.
\end{equation}
Since $E(f)\le E_*<E(W)$, equations \eqref{eq:energy-y} and \eqref{eq:g-one} imply that $y\le1-\delta_0$ for some $\delta_0>0$, which proves \eqref{eq:gradient-gap}.

Equation \eqref{eq:sob-y} may also be written as
\[
 \|f\|_6^6\le y^2\|\nabla f\|_2^2.
\]
Thus
\[
 K(f)
 =\|\nabla f\|_2^2-\|f\|_6^6
 \ge(1-y^2)\|\nabla f\|_2^2
 \ge c_0\|\nabla f\|_2^2,
\]
which is \eqref{eq:K-positive}. The identity
\[
 E(f)=\frac13\|\nabla f\|_2^2+\frac16K(f)
\]
and $K(f)\ge0$ yield \eqref{eq:E-equivalence}. Finally,
\[
 \cS^2\|f\|_6^4
 \le \cS^6\|\nabla f\|_2^4
 =y^2
 \le(1-\delta_0)^2.
\]
After decreasing $\delta_0$, this gives \eqref{eq:L6-gap}.

For a strong solution, conservation of energy and continuity prevent the orbit from crossing $\|\nabla u\|_2=\|\nabla W\|_2$, so the estimates hold at every time.
\end{proof}

The coercivity immediately gives a uniform lower bound for the gradient along any nonzero below-threshold orbit.

\begin{corollary}[Gradient nondegeneracy]\label{cor:gradient-nondegenerate}
Let $u$ be a nonzero strong solution satisfying \eqref{eq:threshold}. Then there exists $a_u>0$ such that
\[
 \|\nabla u(t)\|_2^2\ge a_u
\]
for every $t$ in the lifespan.
\end{corollary}

\begin{proof}
By \eqref{eq:E-equivalence} and conservation of energy,
\[
 \|\nabla u(t)\|_2^2\ge2E(u_0)>0.
\]
If $E(u_0)=0$, then \eqref{eq:E-equivalence} forces $u_0=0$, contradicting nontriviality. Thus one may take $a_u:=2E(u_0)$.
\end{proof}

\subsection{Minimal elements and one-sided normalization}

To make the logical role and range of the external inputs explicit, we formulate the local theory, concentration--compactness statements, and long-time estimates used below as numbered results.

We begin with local well-posedness, stability, and persistence of regularity.

\begin{proposition}[Local theory and stability]\label{prop:local-theory}
For every $u_0\in\dot H^1(\R^3)$, there exists a unique maximal-lifespan strong solution
\[
 u:I_{\max}=(T_-,T_+)\longrightarrow\dot H^1(\R^3),
 \qquad u(0)=u_0,
\]
and the energy is conserved. Moreover:
\begin{enumerate}[label=\textup{(\roman*)},leftmargin=2.4em]
 \item If $T_+<\infty$, then for every $t_0\in I_{\max}$,
 \[
  S_{[t_0,T_+)}(u)=\infty.
 \]
 If $T_+=\infty$ and $S_{[t_0,\infty)}(u)<\infty$, then $u$ scatters forward; the analogous backward statement holds. There exists an absolute constant $\eta_0>0$ such that, whenever $\|u_0\|_{\dot H^1}\le\eta_0$,
 \[
  I_{\max}=\R,
  \qquad S_{\R}(u)\lesssim\|u_0\|_{\dot H^1}.
 \]
 \item If $u_0\in H^1(\R^3)$, then
 \begin{equation}\label{eq:H1-persistence-local-theory}
  u\in C_tH_x^1(I_{\max}\times\R^3),
  \qquad
  \M(u(t))=\M(u_0)\quad(t\in I_{\max}).
 \end{equation}
 More generally, if $u_{0,n}\to u_0$ in $H^1$ and $J\Subset I_{\max}$, then for all sufficiently large $n$, the solution $u_n$ with datum $u_{0,n}$ exists on $J$ and
 \begin{equation}\label{eq:H1-continuous-dependence}
  u_n\longrightarrow u
  \quad\text{in }C_tH_x^1(J\times\R^3)
  \cap L_{t,x}^{10}(J\times\R^3).
 \end{equation}
 \item Let $J\subset\R$ be an interval, $t_0\in J$, and $E,L>0$. Suppose an approximate solution $\widetilde u$ and an error $e$ satisfy
 \[
  (\ii\partial_t+\Delta)\widetilde u+|\widetilde u|^4\widetilde u=e,
  \qquad
  \|\widetilde u\|_{L_t^\infty\dot H_x^1(J)}\le E,
  \qquad
  S_J(\widetilde u)\le L.
 \]
 For every $E'>0$, there exists $\varepsilon_*(E,E',L)>0$ such that, if $0<\varepsilon\le\varepsilon_*(E,E',L)$ and
 \begin{align*}
  \|u_0-\widetilde u(t_0)\|_{\dot H^1}&\le E',\\
  \left\|e^{\ii(t-t_0)\Delta}
  (u_0-\widetilde u(t_0))\right\|_{L_{t,x}^{10}(J\times\R^3)}
  +\|\nabla e\|_{L_t^2L_x^{6/5}(J\times\R^3)}&\le\varepsilon,
 \end{align*}
 then the exact solution with $u(t_0)=u_0$ exists on $J$, and for some $c=c(3)>0$,
 \begin{align}
  S_J(u)+\|\nabla u\|_{L_t^2L_x^6(J\times\R^3)}
  &\le C(E,E',L),
  \label{eq:stability-bound}\\
  \|u-\widetilde u\|_{L_{t,x}^{10}(J\times\R^3)}
  +\|\nabla(u-\widetilde u)\|_{L_t^2L_x^6(J\times\R^3)}
  &\le C(E,E',L)\varepsilon^c.
  \label{eq:stability-closeness}
 \end{align}
\end{enumerate}
\end{proposition}

\begin{proof}
We explain the sources of Proposition~\ref{prop:local-theory}, its applicability under the focusing sign, and how the persistence and continuous-dependence statements used here follow from the standard local theory. The Strichartz estimates are due to Keel and Tao \cite{KeelTao1998}. Maximal-lifespan solutions, the blow-up criterion, scattering criterion, small-data scattering, and unconditional uniqueness are given in \cite[Theorem~3.4 and Corollaries~3.5 and~3.9]{KillipVisan2013}; energy-critical stability is \cite[Theorem~3.8]{KillipVisan2013}. See also \cite{Cazenave2003,KenigMerle2006}. These results are proved for the unified notation
\[
 (\ii\partial_t+\Delta)u=F(u),
 \qquad
 F(u)=\mu|u|^4u,
 \qquad \mu\in\{-1,+1\}.
\]
The present equation corresponds to $\mu=-1$. The local contraction and stability arguments use only
\begin{align}
 |F(z)|&=|z|^5,
 \label{eq:nonlinearity-sign-free-1}\\
 |\nabla F(u)|&\lesssim |u|^4|\nabla u|,\notag\\
 |\nabla(F(u)-F(v))|
 &\lesssim (|u|^4+|v|^4)|\nabla(u-v)|
 \notag\\
 &\quad+(|u|^3+|v|^3)(|\nabla u|+|\nabla v|)|u-v|,
 \label{eq:nonlinearity-sign-free-3}
\end{align}
so the constants are independent of $\mu$. Substitution of \eqref{eq:nonlinearity-sign-free-1}--\eqref{eq:nonlinearity-sign-free-3} into the Strichartz and continuity arguments yields \eqref{eq:stability-bound}--\eqref{eq:stability-closeness}. Persistence of $H^1$ regularity and continuous dependence follow by applying the same local iteration simultaneously to $u$ and $\nabla u$. Applying local mass conservation to smooth approximating solutions and passing to the $C_tH_x^1$ limit gives \eqref{eq:H1-persistence-local-theory}--\eqref{eq:H1-continuous-dependence}.
\end{proof}

On an almost-periodic orbit, the local theory also yields local constancy of the scale and a decomposition into characteristic intervals.

\begin{lemma}[Local constancy]\label{lem:local-constancy}
Let $u:I\times\R^3\to\C$ be a nonzero maximal-lifespan almost-periodic strong solution. There exist $\delta_u,c_u,C_u>0$ such that, for every $t_0\in I$,
\begin{align}
 \bigl[t_0-\delta_uN(t_0)^{-2},
 t_0+\delta_uN(t_0)^{-2}\bigr]
 &\subset I,
 \label{eq:local-lifespan-window}\\
 c_uN(t_0)\le N(t)\le C_uN(t_0)
 &\qquad\bigl(|t-t_0|\le\delta_uN(t_0)^{-2}\bigr).
 \label{eq:local-constancy}
\end{align}
After replacing $N$ by a comparable scale function, the time axis can be partitioned into consecutive characteristic intervals $J_k$, and numbers $N_k>0$ can be chosen such that
\begin{equation}\label{eq:characteristic-interval}
 c_uN_k^{-2}\le |J_k|\le C_uN_k^{-2},
 \qquad
 c_uN_k\le N(t)\le C_uN_k
 \quad(t\in J_k).
\end{equation}
This replacement does not change whether $\int_JN(t)^{-1}\dd t$ is finite. If $T_+:=\sup I<\infty$, then $N(t)\to\infty$ as $t\uparrow T_+$.
\end{lemma}

\begin{proof}
We explain the sources of Lemma~\ref{lem:local-constancy}, the normalization in the present notation, and the characteristic-interval and finite-endpoint conclusions. Local constancy is \cite[Lemma~1.5]{KillipVisan2012}; the general energy-critical formulation is \cite[Lemma~5.18]{KillipVisan2013}. Fix $t_0\in I$ and normalize by
\begin{equation}\label{eq:local-normalized-solution}
 v_{t_0}(s,y)
 :=N(t_0)^{-1/2}
 u\left(t_0+\frac{s}{N(t_0)^2},
 x(t_0)+\frac{y}{N(t_0)}\right).
\end{equation}
The initial-data set
\[
 \{v_{t_0}(0):t_0\in I\}=\mathcal K_u
\]
is precompact in $\dot H^1$. The local theory has a uniform existence time and stability radius on compact sets. Thus there exist $\delta_u>0$ and a fixed precompact set $\mathcal K'\subset\dot H^1$ such that the solution $V_{t_0}$ with datum $v_{t_0}(0)$ satisfies
\[
 V_{t_0}\in C([-\delta_u,\delta_u];\dot H^1(\R^3)),
 \qquad
 \{V_{t_0}(s):t_0\in I,\ |s|\le\delta_u\}\subset\mathcal K'.
\]
Uniqueness in the critical class identifies $V_{t_0}$ with the rescaled solution in \eqref{eq:local-normalized-solution}. Maximality then gives
\[
 [-\delta_u,\delta_u]\subset N(t_0)^2(I-t_0),
\]
which is \eqref{eq:local-lifespan-window}. On the precompact set $\mathcal K'$, any two admissible almost-periodic scales differ by at most a fixed factor. Hence, for $|s|\le\delta_u$,
\[
 c_u\le\frac{N(t_0+s/N(t_0)^2)}{N(t_0)}\le C_u,
\]
and scaling back gives \eqref{eq:local-constancy}.

Fix a sufficiently small constant $\eta_{\mathrm{ch}}>0$, determined by the compact orbit and local theory, and partition the interior intervals by
\[
 S_{J_k}(u)=\eta_{\mathrm{ch}},
\]
merging endpoint remainders in the usual way. By \eqref{eq:local-normalized-solution}, uniform local theory on compact sets, and scale invariance of the scattering size, if $t_k\in J_k$ and $N_k:=N(t_k)$, then
\[
 0<c_u\le N_k^2|J_k|\le C_u<\infty,
 \qquad
 N(t)\simeq_uN_k\quad(t\in J_k).
\]
This is the characteristic-interval construction following \cite[Lemma~1.7]{KillipVisan2012}. Replacing the scale on $J_k$ by $N_k$ changes it only by fixed factors and yields \eqref{eq:characteristic-interval}.

If $T_+<\infty$, the local interval in \eqref{eq:local-lifespan-window} cannot cross the maximal-lifespan endpoint. Hence
\[
 t+\delta_uN(t)^{-2}\le T_+,
\]
so
\[
 N(t)\ge\delta_u^{1/2}(T_+-t)^{-1/2}.
\]
Thus $N(t)\to\infty$ as $t\uparrow T_+$. See also \cite[Corollary~1.6]{KillipVisan2012} and \cite[Corollary~5.19]{KillipVisan2013}.
\end{proof}

Minimal nonscattering solutions also satisfy a Duhamel representation with no free linear remainder.

\begin{proposition}[No-waste formula]\label{prop:no-waste}
Let $u:I\times\R^3\to\C$ be a maximal-lifespan almost-periodic strong solution, and set $T_+:=\sup I$. Assume either $T_+<\infty$, or $T_+=\infty$ and $u$ does not scatter forward. Then for every $t<T_+$,
\begin{equation}\label{eq:no-waste}
 u(t)
 =-\ii\,w\!\!\lim_{T\uparrow T_+}
 \int_t^T e^{\ii(t-s)\Delta}|u(s)|^4u(s)\dd s
 \quad\text{in }\dot H^1,
\end{equation}
where, when $T_+=\infty$, the limit is understood as $T\to\infty$. In particular, in the finite-endpoint case,
\begin{equation}\label{eq:no-waste-finite}
 u(t)
 =-\ii\,w\!\!\lim_{T\uparrow T_+}
 \int_t^T e^{\ii(t-s)\Delta}|u(s)|^4u(s)\dd s
 \quad\text{in }\dot H^1.
\end{equation}
\end{proposition}

\begin{proof}
We explain the source of Proposition~\ref{prop:no-waste} and check the sign change in the Duhamel formula when the defocusing convention in the literature is rewritten with the present focusing sign. The result is \cite[Proposition~1.9]{KillipVisan2012}; the abstract energy-critical form is \cite[Proposition~5.23]{KillipVisan2013}. For $t<T<T_+$, the present Duhamel convention gives
\begin{equation}\label{eq:no-waste-rearranged-duhamel}
 u(t)=e^{\ii(t-T)\Delta}u(T)
 -\ii\int_t^T e^{\ii(t-s)\Delta}|u(s)|^4u(s)\dd s.
\end{equation}
The asymptotic orthogonality argument in those references yields
\[
 e^{-\ii T\Delta}u(T)\rightharpoonup0
 \quad\text{in }\dot H^1
 \qquad(T\uparrow T_+),
\]
and therefore
\[
 e^{\ii(t-T)\Delta}u(T)
 =e^{\ii t\Delta}\bigl(e^{-\ii T\Delta}u(T)\bigr)\rightharpoonup0.
\]
Passing to the limit in \eqref{eq:no-waste-rearranged-duhamel} gives \eqref{eq:no-waste}. The literature often writes the defocusing equation as $\ii u_t+\Delta u=|u|^4u$, whereas \eqref{eq:nls} is $\ii u_t+\Delta u=-|u|^4u$. The sign before the Duhamel integral is therefore reversed, while the weak-convergence argument is unchanged.
\end{proof}

For profile decompositions, write the scaling--translation action as
\[
 (G_{\lambda,x_0}f)(x)
 :=\lambda^{-1/2}f\left(\frac{x-x_0}{\lambda}\right),
 \qquad \lambda>0,\quad x_0\in\R^3,
\]
with inverse
\[
 (G_{\lambda,x_0}^{-1}g)(y)
 :=\lambda^{1/2}g(x_0+\lambda y).
\]
The corresponding transformation of solutions is
\[
 (\mathcal T_{\lambda,x_0}u)(t,x)
 :=\lambda^{-1/2}u\left(\frac{t}{\lambda^2},
 \frac{x-x_0}{\lambda}\right).
\]
These transformations preserve the $\dot H^1$ norm, the energy, and the critical scattering size.

The construction of a minimal critical element requires a combined linear profile decomposition and nonlinear approximation statement.

\begin{proposition}[Nonlinear profile approximation]\label{prop:profile-package}
Let $\{f_n\}$ be bounded in $\dot H^1(\R^3)$. After passing to a subsequence, for every profile index $j\in\mathbb N$ there exist
\[
 \phi^j\in\dot H^1,
 \qquad
 (\lambda_n^j,x_n^j,\tau_n^j)\in(0,\infty)\times\R^3\times\R,
\]
and for every truncation index $J\in\mathbb N$ there is a remainder $w_n^J\in\dot H^1(\R^3)$ such that
\[
 f_n=\sum_{j=1}^J
 G_{\lambda_n^j,x_n^j}e^{\ii\tau_n^j\Delta}\phi^j+w_n^J.
\]
For $j\ne k$, the parameters satisfy
\begin{align}
 &\frac{\lambda_n^j}{\lambda_n^k}
 +\frac{\lambda_n^k}{\lambda_n^j}
 +\frac{|x_n^j-x_n^k|^2}{\lambda_n^j\lambda_n^k}
 +\frac{|(\lambda_n^j)^2\tau_n^j-(\lambda_n^k)^2\tau_n^k|}
 {\lambda_n^j\lambda_n^k}
 \longrightarrow\infty.
 \label{eq:profile-parameter-orthogonality}
\end{align}
After another subsequence and absorption of finite time limits into the profiles, for each $j$,
\begin{equation}\label{eq:profile-time-normalization}
 \tau_n^j\equiv0,
 \qquad\text{or}\qquad
 \tau_n^j\to+\infty,
 \qquad\text{or}\qquad
 \tau_n^j\to-\infty.
\end{equation}
Moreover,
\begin{align}
 \|\nabla f_n\|_2^2
 &=\sum_{j=1}^J\|\nabla\phi^j\|_2^2
   +\|\nabla w_n^J\|_2^2+o_n(1),
 \label{eq:profile-kinetic-decoupling}\\
 E(f_n)
 &=\sum_{j=1}^J
 E\!\left(G_{\lambda_n^j,x_n^j}
 e^{\ii\tau_n^j\Delta}\phi^j\right)
 +E(w_n^J)+o_n(1),
 \label{eq:profile-energy-decoupling}\\
 \lim_{J\to\infty}\limsup_{n\to\infty}
 \|e^{\ii t\Delta}w_n^J\|_{L_{t,x}^{10}(\R\times\R^3)}&=0.
 \label{eq:profile-linear-remainder-small}
\end{align}

Let $U^j:I^j\times\R^3\to\C$ be the associated nonlinear profile: if $\tau_n^j\equiv0$, then $U^j(0)=\phi^j$; if $\tau_n^j\to\pm\infty$, then
\begin{equation}\label{eq:nonlinear-profile-wave-operator}
 \|U^j(t)-e^{\ii t\Delta}\phi^j\|_{\dot H^1}
 \longrightarrow0
 \qquad(t\to\pm\infty).
\end{equation}
Define
\[
 U_n^j(t,x)
 :=(\lambda_n^j)^{-1/2}
 U^j\left(\frac{t}{(\lambda_n^j)^2}+\tau_n^j,
 \frac{x-x_n^j}{\lambda_n^j}\right).
\]
Then, for each fixed $j$,
\begin{equation}\label{eq:nonlinear-profile-initial-agreement}
 \|U_n^j(0)-G_{\lambda_n^j,x_n^j}
 e^{\ii\tau_n^j\Delta}\phi^j\|_{\dot H^1}
 \longrightarrow0.
\end{equation}

Let $u_n:I_n\times\R^3\to\C$ be the exact solution with initial datum $f_n$ at $t=0$, and set
\begin{align*}
 I_n^+&:=I_n\cap[0,\infty),
 &I_n^-&:=I_n\cap(-\infty,0],\\
 I_{n,j}^\sigma
 &:=\tau_n^j+(\lambda_n^j)^{-2}I_n^\sigma,
 &&\sigma\in\{+,-\}.
\end{align*}
If $I_{n,j}^\sigma\not\subset I^j$, we set $S_{I_{n,j}^\sigma}(U^j)=\infty$. Fix $\sigma\in\{+,-\}$. Suppose there exists $J_0\in\mathbb N$ such that
\begin{align}
 &\|\nabla\phi^j\|_2\le\eta_0
 \qquad(j\ge J_0),
 \label{eq:profile-small-tail}\\
 &\max_{1\le j<J_0}
 \limsup_{n\to\infty}S_{I_{n,j}^\sigma}(U^j)<\infty,
 \label{eq:large-profile-bounds}
\end{align}
where $\eta_0$ is the small-data threshold; when $J_0=1$, \eqref{eq:large-profile-bounds} is interpreted as an empty condition. Then
\begin{equation}\label{eq:profile-stability-consequence}
 \limsup_{n\to\infty}S_{I_n^\sigma}(u_n)<\infty.
\end{equation}
\end{proposition}

\begin{proof}
We explain the sources of the linear and nonlinear profile decompositions used in Proposition~\ref{prop:profile-package} and the adaptation of the nonlinear approximation to the present focusing sign and energy-threshold notation. The linear decomposition, \eqref{eq:profile-parameter-orthogonality}, \eqref{eq:profile-kinetic-decoupling}, and \eqref{eq:profile-linear-remainder-small} are due to Keraani \cite{Keraani2001}; see also \cite[Theorem~4.27]{KillipVisan2013}. Brezis--Lieb decoupling of the potential energy \cite{BrezisLieb1983}, together with \eqref{eq:profile-kinetic-decoupling}, gives \eqref{eq:profile-energy-decoupling}. The wave operators follow from the small-tail theory in Proposition~\ref{prop:local-theory}, and \eqref{eq:nonlinear-profile-initial-agreement} follows directly from \eqref{eq:nonlinear-profile-wave-operator}.

Under \eqref{eq:profile-small-tail}--\eqref{eq:large-profile-bounds}, the profiles $U_n^j$ are defined on $I_n^\sigma$ for all sufficiently large $n$. Put
\[
 \widetilde u_n^J
 :=\sum_{j=1}^J U_n^j+e^{\ii t\Delta}w_n^J.
\]
Parameter orthogonality, spacetime decoupling of nonlinear profiles, and linear Strichartz smallness of the remainder yield
\begin{align}
 &\lim_{J\to\infty}\limsup_{n\to\infty}
 \left\|\nabla\left[
 (\ii\partial_t+\Delta)\widetilde u_n^J
 +|\widetilde u_n^J|^4\widetilde u_n^J
 \right]\right\|_{L_t^2L_x^{6/5}(I_n^\sigma\times\R^3)}=0,
 \label{eq:nonlinear-profile-error-small}\\
 &\lim_{J\to\infty}\limsup_{n\to\infty}
 \|\widetilde u_n^J(0)-f_n\|_{\dot H^1}=0,\notag\\
 &\sup_J\limsup_{n\to\infty}
 \left[
 S_{I_n^\sigma}(\widetilde u_n^J)
 +\|\widetilde u_n^J\|_{L_t^\infty\dot H_x^1(I_n^\sigma)}
 \right]<\infty.
 \label{eq:nonlinear-profile-uniform-bounds}
\end{align}
These are precisely the three inputs of the nonlinear profile approximation in \cite[Lemmas~5.7--5.10 and the proof of Proposition~5.6]{KillipVisan2013}. Applying Proposition~\ref{prop:local-theory}(iii) to \eqref{eq:nonlinear-profile-error-small}--\eqref{eq:nonlinear-profile-uniform-bounds}, first fixing a sufficiently large $J$ and then letting $n\to\infty$, gives \eqref{eq:profile-stability-consequence}.

The cited analysis is carried out for
\[
 (\ii\partial_t+\Delta)u=\mu|u|^4u,
 \qquad \mu\in\{-1,+1\},
\]
and the present equation corresponds to $\mu=-1$. The error estimates use only \eqref{eq:nonlinearity-sign-free-1}--\eqref{eq:nonlinearity-sign-free-3}. The only additional focusing requirement is that every profile declared subcritical lies and remains on the stable below-threshold branch; this follows from Lemma~\ref{lem:coercivity}, conservation of energy, and continuity.
\end{proof}

The profile approximation yields compactness at the critical energy.

\begin{proposition}[Palais--Smale compactness]\label{prop:palais-smale-energy}
Assume that below-threshold scattering fails, and let $E_c$ be defined by \eqref{eq:critical-energy}. Let $u_n:I_n\times\R^3\to\C$ be maximal-lifespan strong solutions and $t_n\in I_n$ such that
\begin{align}
 u_n(t_n)&\in\mathcal A,
 \qquad E(u_n)\longrightarrow E_c,
 \label{eq:PS-energy-assumptions}\\
 S_{I_n\cap(-\infty,t_n]}(u_n)&\longrightarrow\infty,
 \qquad
 S_{I_n\cap[t_n,\infty)}(u_n)\longrightarrow\infty.
 \label{eq:PS-directional-blowup}
\end{align}
Then, after passing to a subsequence, there exist $\lambda_n>0$, $x_n\in\R^3$, and $f\in\dot H^1(\R^3)$ such that
\begin{equation}\label{eq:PS-strong-conclusion}
 G_{\lambda_n,x_n}^{-1}u_n(t_n)
 \longrightarrow f
 \qquad\text{in }\dot H^1(\R^3).
\end{equation}
\end{proposition}

\begin{proof}
After a time translation, assume $t_n=0$ and write
\[
 I_n^-:=I_n\cap(-\infty,0],
 \qquad
 I_n^+:=I_n\cap[0,\infty).
\]
Choose
\[
 E_c<E^\sharp<E(W).
\]
By \eqref{eq:PS-energy-assumptions}, conservation of energy, and Lemma~\ref{lem:coercivity}, after discarding finitely many indices there is $\delta>0$ such that
\begin{equation}\label{eq:PS-uniform-gradient-gap}
 \sup_n\sup_{t\in I_n}\|\nabla u_n(t)\|_2^2
 \le(1-\delta)\|\nabla W\|_2^2.
\end{equation}

Apply Proposition~\ref{prop:profile-package} to $f_n:=u_n(0)$, and take a diagonal subsequence so that
\[
 E^j
 :=\lim_{n\to\infty}
 E\!\left(G_{\lambda_n^j,x_n^j}
 e^{\ii\tau_n^j\Delta}\phi^j\right)
\]
exists for every $j$. If $\tau_n^j\equiv0$, then $E^j=E(\phi^j)=E(U^j)$. If $\tau_n^j\to\pm\infty$, density in $\dot H^1$ and free dispersion give
\[
 \|e^{\ii t\Delta}\phi^j\|_6\longrightarrow0
 \qquad(|t|\to\infty).
\]
Together with \eqref{eq:nonlinear-profile-wave-operator} and conservation of energy, this yields
\[
 E^j=\frac12\|\nabla\phi^j\|_2^2=E(U^j).
\]
Thus $E^j=E(U^j)$ for all $j$.

By \eqref{eq:profile-kinetic-decoupling} and \eqref{eq:PS-uniform-gradient-gap},
\begin{equation}\label{eq:PS-individual-gradient-gap}
 \|\nabla\phi^j\|_2^2
 \le(1-\delta)\|\nabla W\|_2^2
 \qquad(j\ge1),
\end{equation}
and, for each fixed $J$ and all sufficiently large $n$,
\[
 \|\nabla w_n^J\|_2^2
 \le(1-\delta/2)\|\nabla W\|_2^2.
\]
The sharp Sobolev inequality therefore gives $c_\delta>0$ such that
\begin{align}
 E^j&\ge c_\delta\|\nabla\phi^j\|_2^2\ge0,
 \label{eq:PS-profile-energy-positive}\\
 E(w_n^J)&\ge c_\delta\|\nabla w_n^J\|_2^2\ge0
 \qquad(n\gg_J1).
 \label{eq:PS-remainder-energy-positive}
\end{align}
Letting first $n\to\infty$ in \eqref{eq:profile-energy-decoupling} and then $J\to\infty$, we obtain
\[
 \sum_{j\ge1}E^j\le E_c.
\]

We next identify a profile responsible for forward divergence. By \eqref{eq:profile-kinetic-decoupling}, there is $J_0\in\mathbb N$ such that
\[
 \|\nabla\phi^j\|_2\le\eta_0
 \qquad(j\ge J_0),
\]
where $\eta_0$ is the small-data threshold. Call an index $1\le j<J_0$ a forward bad profile if
\[
 \limsup_{n\to\infty}S_{I_{n,j}^+}(U^j)=\infty.
\]
If no forward bad profile existed, then
\begin{equation}\label{eq:PS-all-large-profiles-good}
 \max_{1\le j<J_0}
 \limsup_{n\to\infty}S_{I_{n,j}^+}(U^j)<\infty,
\end{equation}
and Proposition~\ref{prop:profile-package} would give
\[
 \limsup_{n\to\infty}S_{I_n^+}(u_n)<\infty,
\]
contradicting \eqref{eq:PS-directional-blowup}. This also covers the case in which every profile vanishes, because \eqref{eq:PS-all-large-profiles-good} is then an empty condition. Hence there is $j_0$ such that
\begin{equation}\label{eq:PS-forward-bad-profile}
 \limsup_{n\to\infty}S_{I_{n,j_0}^+}(U^{j_0})=\infty.
\end{equation}

If $E^{j_0}<E_c$, choose $E^{j_0}<E_*<E_c$. When $\tau_n^{j_0}\equiv0$, take $t_*=0$; when $\tau_n^{j_0}\to\pm\infty$, choose $t_*$ sufficiently far in the corresponding scattering direction using \eqref{eq:nonlinear-profile-wave-operator}. In conjunction with \eqref{eq:PS-individual-gradient-gap}, we may arrange
\[
 E(U^{j_0}(t_*))=E^{j_0}\le E_*,
 \qquad
 \|\nabla U^{j_0}(t_*)\|_2<\|\nabla W\|_2.
\]
Thus $U^{j_0}(t_*)\in\mathcal A$. By \eqref{eq:critical-energy-range} and \eqref{eq:uniform-scattering-size}, the maximal-lifespan solution with this datum is global and satisfies
\[
 S(U^{j_0})\le\mathcal L(E_*)<\infty.
\]
Uniqueness in the critical class identifies it with the original nonlinear profile, contradicting \eqref{eq:PS-forward-bad-profile}. Therefore $E^{j_0}=E_c$.

Relabel $j_0=1$. Taking $J=1$ in \eqref{eq:profile-energy-decoupling}, and using \eqref{eq:PS-remainder-energy-positive} and $E^1=E_c$, gives
\begin{align*}
 E(w_n^1)
 &=E(f_n)-E\!\left(G_{\lambda_n^1,x_n^1}
 e^{\ii\tau_n^1\Delta}\phi^1\right)+o_n(1)
 \longrightarrow0,\\
 \|\nabla w_n^1\|_2^2
 &\le c_\delta^{-1}E(w_n^1)
 \longrightarrow0.
\end{align*}
Setting
\[
 \lambda_n:=\lambda_n^1,
 \quad x_n:=x_n^1,
 \quad \tau_n:=\tau_n^1,
 \quad \phi:=\phi^1,
 \quad w_n:=w_n^1,
\]
we have
\begin{equation}\label{eq:PS-single-profile-data}
 f_n=G_{\lambda_n,x_n}e^{\ii\tau_n\Delta}\phi+w_n,
 \qquad
 \|w_n\|_{\dot H^1}\to0,
 \qquad \phi\ne0.
\end{equation}

By \eqref{eq:profile-time-normalization}, it remains to exclude $\tau_n\to\pm\infty$. If $\tau_n\to+\infty$, scale invariance and Strichartz estimates give
\[
 \|e^{\ii t\Delta}f_n\|_{L_{t,x}^{10}([0,\infty)\times\R^3)}
 \le
 \|e^{\ii t\Delta}\phi\|_{L_{t,x}^{10}([\tau_n,\infty)\times\R^3)}
 +C\|w_n\|_{\dot H^1}
 \longrightarrow0.
\]
Applying Proposition~\ref{prop:local-theory}(iii) with the zero solution as approximate solution gives $S_{I_n^+}(u_n)\to0$, contradicting \eqref{eq:PS-directional-blowup}. If $\tau_n\to-\infty$, similarly,
\begin{align*}
 \|e^{\ii t\Delta}f_n\|_{L_{t,x}^{10}(( -\infty,0]\times\R^3)}
 &\le
 \|e^{\ii t\Delta}\phi\|_{L_{t,x}^{10}(( -\infty,\tau_n]\times\R^3)}
 +C\|w_n\|_{\dot H^1}
 \longrightarrow0,\\
 S_{I_n^-}(u_n)&\longrightarrow0,
\end{align*}
again a contradiction. Thus the finite time limit may be absorbed into $\phi$, and we take $\tau_n\equiv0$. Equation \eqref{eq:PS-single-profile-data} then yields
\[
 G_{\lambda_n,x_n}^{-1}f_n
 =\phi+G_{\lambda_n,x_n}^{-1}w_n
 \longrightarrow\phi
 \quad\text{in }\dot H^1,
\]
which is \eqref{eq:PS-strong-conclusion}.

This is the energy-threshold version of \cite[Proposition~5.6]{KillipVisan2013}. The reference is parametrized by a critical kinetic-energy threshold. Here \eqref{eq:PS-uniform-gradient-gap}, \eqref{eq:PS-profile-energy-positive}, and \eqref{eq:PS-remainder-energy-positive} keep every profile and remainder on the stable below-threshold branch, allowing the conserved energy $E_c$ to replace the nonconserved kinetic threshold. The remaining bad-profile, nonlinear-approximation, and time-parameter arguments are the same.
\end{proof}

Combining Palais--Smale compactness with a minimizing sequence produces the critical element required below.

\begin{theorem}[Minimal critical element]\label{thm:minimal-element}
If not every datum satisfying \eqref{eq:threshold} generates a global solution scattering in both time directions, then there exists a maximal-lifespan strong solution
\[
 u_c:I_c=(T_-,T_+)\longrightarrow\dot H^1(\R^3),
 \qquad 0\in I_c,
\]
with the following properties:
\begin{enumerate}[label=\textup{(\roman*)},leftmargin=2.4em]
 \item
 \begin{equation}\label{eq:minimal-energy-and-branch}
  E(u_c)=E_c,
  \qquad
  0<E_c<E(W),
  \qquad
  \|\nabla u_c(t)\|_2<\|\nabla W\|_2
  \quad(t\in I_c);
 \end{equation}
 \item $u_c$ is almost periodic modulo scaling and translation;
 \item the critical scattering size diverges toward both lifespan endpoints:
 \[
  S_{(T_-,0]}(u_c)
  =S_{[0,T_+)}(u_c)
  =\infty;
 \]
 \item $u_c$ satisfies the local constancy statement of Lemma~\ref{lem:local-constancy} and the no-waste Duhamel formula of Proposition~\ref{prop:no-waste}.
\end{enumerate}
\end{theorem}

\begin{proof}
By \eqref{eq:critical-energy} and monotonicity of $\mathcal L$, choose
\[
 E_n\downarrow E_c,
 \qquad
 \mathcal L(E_n)=\infty.
\]
There are therefore $a_n\in\mathcal A$ whose maximal-lifespan solutions $u_n:I_n\times\R^3\to\C$ satisfy
\[
 E(a_n)\le E_n,
 \qquad
 S_{I_n}(u_n)\ge2n.
\]
By conservation of energy and invariance of the stable branch,
\[
 u_n(t)\in\mathcal A
 \qquad(t\in I_n).
\]
If, along a subsequence, $E(a_n)\le E_c-\varepsilon$, then
\[
 2n
 \le S_{I_n}(u_n)
 \le\mathcal L(E_c-\varepsilon)
 <\infty,
\]
a contradiction. Hence
\begin{equation}\label{eq:minimizing-energy-convergence}
 E(a_n)\longrightarrow E_c.
\end{equation}

Since $S_{I_n}(u_n)\ge2n$, monotone convergence provides compact intervals $[\alpha_n,\beta_n]\Subset I_n$ such that
\[
 S_{[\alpha_n,\beta_n]}(u_n)\ge2n.
\]
On each interval define
\[
 F_n(t):=S_{[\alpha_n,t]}(u_n)^{10}
 =\int_{\alpha_n}^t\!\int_{\R^3}|u_n(s,x)|^{10}\dd x\dd s.
\]
The function $F_n$ is continuous, with $F_n(\alpha_n)=0$ and $F_n(\beta_n)\ge(2n)^{10}$. The intermediate value theorem gives $t_n\in[\alpha_n,\beta_n]$ such that
\begin{align*}
 S_{[\alpha_n,t_n]}(u_n)&=n,\\
 S_{[t_n,\beta_n]}(u_n)^{10}
 &=S_{[\alpha_n,\beta_n]}(u_n)^{10}-n^{10}
 \ge(2^{10}-1)n^{10}.
\end{align*}
Thus
\begin{align}
 S_{I_n\cap(-\infty,t_n]}(u_n)&\ge n,
 \label{eq:split-past-size}\\
 S_{I_n\cap[t_n,\infty)}(u_n)
 &\ge(2^{10}-1)^{1/10}n.
 \label{eq:split-future-size}
\end{align}
Set $f_n:=u_n(t_n)$. Conservation of energy, \eqref{eq:minimizing-energy-convergence}, and \eqref{eq:split-past-size}--\eqref{eq:split-future-size} place $(u_n,t_n)$ under Proposition~\ref{prop:palais-smale-energy}. Hence there exist $\lambda_n>0$, $x_n\in\R^3$, and $u_0\in\dot H^1$ such that
\[
 G_{\lambda_n,x_n}^{-1}f_n
 \longrightarrow u_0
 \quad\text{in }\dot H^1.
\]
Continuity of the energy on $\dot H^1$ gives $E(u_0)=E_c$. The uniform below-threshold gap in the proof of Proposition~\ref{prop:palais-smale-energy} gives $\delta_c>0$ such that
\[
 \|\nabla u_0\|_2^2
 \le(1-\delta_c)\|\nabla W\|_2^2.
\]
Let $u_c:I_c\times\R^3\to\C$ be the maximal-lifespan solution with $u_c(0)=u_0$. Conservation of energy, continuity, and Lemma~\ref{lem:coercivity} yield \eqref{eq:minimal-energy-and-branch}.

To pass the divergence of scattering size to both directions, define
\[
 \widetilde u_n(t,x)
 :=\lambda_n^{1/2}
 u_n\bigl(t_n+\lambda_n^2t,\,x_n+\lambda_n x\bigr),
 \qquad
 \widetilde I_n
 :=\{t:t_n+\lambda_n^2t\in I_n\}.
\]
Then
\begin{equation}\label{eq:normalized-initial-convergence}
 \widetilde u_n(0)=G_{\lambda_n,x_n}^{-1}f_n
 \longrightarrow u_c(0)
 \quad\text{in }\dot H^1,
\end{equation}
and the critical scattering size is invariant under this transformation.

Suppose $S_{[0,T_+)}(u_c)<\infty$. Proposition~\ref{prop:local-theory}(i) implies $T_+=\infty$. Choose a stability smallness parameter $\eta>0$ and partition $[0,\infty)$ into finitely many consecutive intervals $J_1,\dots,J_L$ such that
\[
 [0,\infty)=\bigcup_{\ell=1}^LJ_\ell,
 \qquad
 S_{J_\ell}(u_c)\le\eta
 \quad(1\le\ell\le L).
\]
Using \eqref{eq:normalized-initial-convergence}, apply Proposition~\ref{prop:local-theory}(iii) successively on these intervals. For all sufficiently large $n$,
\[
 [0,\infty)\subset\widetilde I_n,
 \qquad
 S_{[0,\infty)}(\widetilde u_n)
 \le C\bigl(S_{[0,\infty)}(u_c)\bigr)<\infty.
\]
But scale invariance and \eqref{eq:split-future-size} give
\[
 S_{\widetilde I_n\cap[0,\infty)}(\widetilde u_n)
 \ge(2^{10}-1)^{1/10}n,
\]
a contradiction. Hence $S_{[0,T_+)}(u_c)=\infty$. Repeating the argument backward in time and using \eqref{eq:split-past-size} gives
\begin{equation}\label{eq:minimal-two-sided-blowup}
 S_{(T_-,0]}(u_c)
 =S_{[0,T_+)}(u_c)
 =\infty.
\end{equation}

We now prove almost periodicity. Let $\tau_n\in I_c$ be arbitrary. Since the scattering size is finite on every compact time interval, \eqref{eq:minimal-two-sided-blowup} implies
\[
 S_{(T_-,\tau_n]}(u_c)
 =S_{[\tau_n,T_+)}(u_c)
 =\infty.
\]
Apply Proposition~\ref{prop:palais-smale-energy} at time zero to the translated solutions $v_n(t):=u_c(t+\tau_n)$. After passing to a subsequence, there exist $\mu_n>0$, $y_n\in\R^3$, and $f\in\dot H^1$ such that
\[
 G_{\mu_n,y_n}^{-1}u_c(\tau_n)
 \longrightarrow f
 \quad\text{in }\dot H^1.
\]
Thus every sequence in the orbit can be made strongly convergent by scaling and translation. By the compactness criterion on the quotient space \cite[Proposition~5.6 and Theorem~5.12]{KillipVisan2013}, one can select $N:I_c\to(0,\infty)$ and $x:I_c\to\R^3$ so that
\[
 \left\{
 N(t)^{-1/2}u_c\left(t,x(t)+\frac{\cdot}{N(t)}\right):
 t\in I_c
 \right\}
\]
is precompact in $\dot H^1$. This proves (ii). Finally, (ii) and \eqref{eq:minimal-two-sided-blowup} allow us to apply Lemma~\ref{lem:local-constancy} and Proposition~\ref{prop:no-waste}, giving (iv).
\end{proof}

One-sided normalization also requires upgrading a sequentially available scale lower bound to a modulation choice on the full lifespan.

\begin{lemma}[Selection with a scale lower bound]\label{lem:lower-bounded-modulation-selection}
Let $I\subset\R$ be an interval and $v:I\to\dot H^1(\R^3)$ satisfy $\inf_{s\in I}\|v(s)\|_{\dot H^1}>0$. Let $\mathcal K\subset\dot H^1(\R^3)$ be compact with $0\notin\mathcal K$, and let $N_0>0$. Suppose that for every sequence $s_n\in I$, one can pass to a subsequence, still denoted $s_n$, and find $N_n\ge N_0$, $y_n\in\R^3$, and $h\in\mathcal K$ such that
\begin{equation}\label{eq:modulation-selection-sequential}
 N_n^{-1/2}v\left(s_n,y_n+\frac{\cdot}{N_n}\right)
 \longrightarrow h
 \quad\text{in }\dot H^1.
\end{equation}
Then there exist functions
\[
 N_v:I\to[N_0,\infty),
 \qquad
 x_v:I\to\R^3,
\]
such that
\begin{equation}\label{eq:modulation-selected-orbit}
 \left\{
 N_v(s)^{-1/2}v\left(s,x_v(s)+\frac{\cdot}{N_v(s)}\right):s\in I
 \right\}
 \subset\mathcal K.
\end{equation}
In particular, $v$ is almost periodic modulo scaling and translation, and the selected frequency scale satisfies $N_v(s)\ge N_0$ uniformly.
\end{lemma}

\begin{proof}
Fix $s\in I$ and apply the hypothesis to the constant sequence $s_n=s$. Put $\lambda_n:=N_n^{-1}$, so $0<\lambda_n\le N_0^{-1}$, and write \eqref{eq:modulation-selection-sequential} as
\[
 G_{\lambda_n,y_n}^{-1}v(s)\longrightarrow h\ne0
 \quad\text{in }\dot H^1.
\]
The standard properness property of the critical scaling--translation group says that, for fixed nonzero $f\in\dot H^1$, if $\lambda_n\to0$, $\lambda_n\to\infty$, or $|y_n|\to\infty$ while $\lambda_n$ remains in a compact subset of $(0,\infty)$, then $G_{\lambda_n,y_n}^{-1}f\rightharpoonup0$ in $\dot H^1$. This follows first for $C_c^\infty$ functions by a change of variables and then for all $\dot H^1$ functions by density. Since the strong limit above is nonzero, the parameters $(\lambda_n,y_n)$ remain in a compact subset of $(0,\infty)\times\R^3$. After passing to a subsequence,
\[
 \lambda_n\to\lambda_s\in(0,N_0^{-1}],
 \qquad
 y_n\to x_s\in\R^3.
\]
Strong continuity of the group action in the parameters gives
\[
 G_{\lambda_s,x_s}^{-1}v(s)=h\in\mathcal K.
\]
Set $N_v(s):=\lambda_s^{-1}\ge N_0$, and choose one such pair $(N_v(s),x_v(s))$ for every $s\in I$. This gives \eqref{eq:modulation-selected-orbit}; since the orbit lies in the compact set $\mathcal K$, its closure is compact.
\end{proof}

The selection lemma permits us to reselect a one-sided normalized representative of the minimal compactness class.

\begin{proposition}[One-sided normalization]\label{prop:one-sided-selection}
If below-threshold scattering fails, then, after applying time reversal if necessary, one may reselect a nonzero almost-periodic strong solution of energy $E_c$,
\[
 u_c:[0,T_+)\times\R^3\to\C,
\]
such that $S_{[0,T_+)}(u_c)=\infty$ and exactly one of the following alternatives holds:
\begin{equation}\label{eq:two-normalized-cases}
 T_+<\infty,
 \qquad\text{or}\qquad
 T_+=\infty\ \text{and}\ N(t)\ge1
 \quad(t\ge0).
\end{equation}
The reselected solution remains on the stable below-threshold branch and retains almost periodicity, local constancy, and the no-waste Duhamel formula.
\end{proposition}

\begin{proof}
Start with the minimal element $u_*:I_*\times\R^3\to\C$ from Theorem~\ref{thm:minimal-element}, together with fixed almost-periodic parameters $N_*:I_*\to(0,\infty)$ and $x_*:I_*\to\R^3$. Divergence of the scattering size toward both endpoints provides compact intervals $[a_n,b_n]\Subset I_*$ such that
\[
 S_{[a_n,b_n]}(u_*)\ge2n.
\]
Local constancy gives a positive lower bound for $N_*$ on every compact time interval. Choose $t_n\in[a_n,b_n]$ satisfying
\[
 N_*(t_n)\le2\inf_{t\in[a_n,b_n]}N_*(t).
\]
Additivity of the tenth power of the scattering size gives
\begin{equation}\label{eq:selection-one-large-side}
 \max\bigl\{S_{[a_n,t_n]}(u_*),S_{[t_n,b_n]}(u_*)\bigr\}
 \ge2^{-1/10}S_{[a_n,b_n]}(u_*)\ge n.
\end{equation}
After passing to a subsequence, the same side realizes \eqref{eq:selection-one-large-side} for every $n$; if it is the left side, replace $u_*$ by its time reversal $u(t,x)\mapsto\overline{u(-t,x)}$. Thus we may assume
\[
 S_{[t_n,b_n]}(u_*)\ge n,
 \qquad
 N_*(t)\ge\frac12N_*(t_n)
 \quad(t_n\le t\le b_n).
\]

Let $N_n:=N_*(t_n)$, $x_n:=x_*(t_n)$, and define
\begin{align*}
 v_n(s,y)
 &:=N_n^{-1/2}
 u_*\left(t_n+\frac{s}{N_n^2},x_n+\frac{y}{N_n}\right),\\
 \beta_n&:=N_n^2(b_n-t_n).
\end{align*}
Then
\begin{equation}\label{eq:selection-rescaled-properties}
 v_n(0)\in\mathcal K_{u_*},
 \qquad
 S_{[0,\beta_n]}(v_n)\ge n,
 \qquad
 N_{v_n}(s):=\frac{N_*(t_n+s/N_n^2)}{N_n}\ge\frac12
 \quad(0\le s\le\beta_n).
\end{equation}
By precompactness of $\mathcal K_{u_*}$, after a subsequence
\[
 v_n(0)\longrightarrow v_0
 \quad\text{in }\dot H^1.
\]
Let $v:[0,T_+(v))\times\R^3\to\C$ be the maximal-lifespan solution with $v(0)=v_0$. Continuity and scaling invariance of the energy give
\[
 E(v)=E_c,
 \qquad
 \|\nabla v(t)\|_2<\|\nabla W\|_2.
\]
After another subsequence, $\beta_n\to\beta\in(0,\infty]$. Uniform local theory on the compact initial-data set gives $\delta_*>0$ and $C_*<\infty$ such that
\[
 [0,\delta_*]\subset I(v_n),
 \qquad
 S_{[0,\delta_*]}(v_n)\le C_*
 \quad(n\ge1).
\]
By \eqref{eq:selection-rescaled-properties}, $\beta_n>\delta_*$ whenever $n>C_*$, so $\beta\ge\delta_*>0$. For every
\[
 0<T<\min\{\beta,T_+(v)\},
\]
critical stability gives
\[
 [0,T]\subset I(v_n)\quad(n\gg_T1),
 \qquad
 v_n\longrightarrow v
 \quad\text{in }C([0,T];\dot H^1(\R^3))\cap L_{t,x}^{10}([0,T]\times\R^3).
\]

We first compare $T_+(v)$ with $\beta$. If $\beta<\infty$ and $T_+(v)>\beta$, choose $\delta>0$ so that
\[
 [0,\beta+2\delta]\Subset I(v),
 \qquad
 S_{[0,\beta+2\delta]}(v)<\infty.
\]
For large $n$, $\beta_n<\beta+\delta$. Long-time stability for $v_n$ around $v$ on $[0,\beta+\delta]$ then gives
\[
 \sup_{n\gg1}S_{[0,\beta_n]}(v_n)<\infty,
\]
contradicting \eqref{eq:selection-rescaled-properties}. Hence
\begin{equation}\label{eq:selection-lifespan-upper}
 T_+(v)\le\beta
 \qquad(\beta<\infty).
\end{equation}
If $T_+(v)<\infty$, the blow-up criterion yields $S_{[0,T_+(v))}(v)=\infty$. If $T_+(v)=\beta=\infty$ but $S_{[0,\infty)}(v)<\infty$, long-time stability on the entire half-line and $\beta_n\to\infty$ again imply
\[
 \sup_{n\gg1}S_{[0,\beta_n]}(v_n)<\infty,
\]
a contradiction. Thus, with $T_+:=T_+(v)$,
\begin{equation}\label{eq:selection-limit-blowup}
 S_{[0,T_+)}(v)=\infty.
\end{equation}

It remains to transfer almost periodicity and the scale lower bound. For $0\le s\le\beta_n$, set
\begin{align*}
 \mathcal N_n(s)
 &:=\frac{N_*(t_n+s/N_n^2)}{N_n}\ge\frac12,\\
 y_n(s)
 &:=N_n\bigl(x_*(t_n+s/N_n^2)-x_n\bigr).
\end{align*}
Then
\begin{equation}\label{eq:selection-common-compactness-class}
 \mathcal N_n(s)^{-1/2}
 v_n\left(s,y_n(s)+\frac{\cdot}{\mathcal N_n(s)}\right)
 \in\mathcal K_{u_*}.
\end{equation}
Let $s_k\in[0,T_+)$ be arbitrary. If $\beta<\infty$, then \eqref{eq:selection-lifespan-upper} gives $s_k<T_+\le\beta$; if $\beta=\infty$, this is automatic. We may choose $n_k\to\infty$ so that
\begin{equation}\label{eq:selection-diagonal-approximation}
 s_k<\beta_{n_k},
 \qquad
 \|v_{n_k}(s_k)-v(s_k)\|_{\dot H^1}\le k^{-1}.
\end{equation}
Set
\[
 \widetilde N_k:=\mathcal N_{n_k}(s_k)\ge\frac12,
 \qquad
 \widetilde x_k:=y_{n_k}(s_k).
\]
Since scaling and translation preserve the $\dot H^1$ distance, \eqref{eq:selection-common-compactness-class}--\eqref{eq:selection-diagonal-approximation} give
\[
 \operatorname{dist}_{\dot H^1}\!\left(
 \widetilde N_k^{-1/2}v\left(s_k,\widetilde x_k+\frac{\cdot}{\widetilde N_k}\right),
 \mathcal K_{u_*}\right)\le k^{-1}.
\]
Because $E(h)=E_c>0$, the compact set $\mathcal K_{u_*}$ does not contain zero, and below-threshold coercivity gives
\[
 \inf_{0\le s<T_+}\|v(s)\|_{\dot H^1}>0.
\]
For every time sequence, the last distance estimate and compactness of $\mathcal K_{u_*}$ permit a further subsequence whose normalized functions converge strongly to an element of $\mathcal K_{u_*}$. Applying Lemma~\ref{lem:lower-bounded-modulation-selection} with $I=[0,T_+)$, $N_0=1/2$, and $\mathcal K=\mathcal K_{u_*}$ shows that $v$ is almost periodic modulo scaling and translation and that its scale can be chosen so that
\begin{equation}\label{eq:selection-limit-scale-lower}
 N_v(s)\ge\frac12
 \qquad(0\le s<T_+).
\end{equation}
This is the one-sided rescaling argument preceding \cite[Theorem~1.8]{KillipVisan2012}; see also \cite[Section~5.3]{KillipVisan2013}.

If $T_+<\infty$, we are in the first alternative of \eqref{eq:two-normalized-cases}. If $T_+=\infty$, apply one fixed scaling to normalize \eqref{eq:selection-limit-scale-lower} to $N_v(s)\ge1$. Finally, \eqref{eq:selection-limit-blowup} and almost periodicity allow the use of Proposition~\ref{prop:no-waste}, and Lemma~\ref{lem:local-constancy} also applies.

This procedure reselects a representative from the minimal compactness class. It does not assert that an arbitrarily prescribed minimal element can be made to satisfy a global frequency lower bound by one fixed scaling. The selection uses only symmetries, compactness, and critical stability and is independent of the defocusing sign.
\end{proof}

\subsection{Consequences of long-time Strichartz estimates}

The exclusion of rapid cascades uses two consequences of the long-time Strichartz recursion.

\begin{proposition}[Long-time Strichartz consequences]\label{prop:LTS}
Let $u:[0,T_{\max})\times\R^3\to\C$ be a maximal-lifespan almost-periodic strong solution with frequency scale $N:[0,T_{\max})\to(0,\infty)$. Assume
\[
 N(t)\ge1,
 \qquad
 S_{[0,T_{\max})}(u)=\infty,
\]
and that the no-waste Duhamel formula in Proposition~\ref{prop:no-waste} holds. For $0<T<T_{\max}$, set
\[
 \mathsf K([0,T]):=\int_0^T N(t)^{-1}\dd t,
 \qquad
 \mathsf K([0,T_{\max})):=\sup_{0<T<T_{\max}}\mathsf K([0,T]).
\]
If
\[
 \mathsf K([0,T_{\max}))<\infty,
\]
then
\begin{equation}\label{eq:LTS-mass-recovery}
 u\in L_t^\infty L_x^2([0,T_{\max})\times\R^3),
\end{equation}
and there exists $C_u<\infty$ such that, for every $0<M\le1$,
\begin{equation}\label{eq:absolute-low-mass}
 \|P_{\le M}u\|_{L_t^\infty L_x^2([0,T_{\max})\times\R^3)}
 \le C_uM^{1/2}.
\end{equation}
\end{proposition}

\begin{proof}
We explain the source of Proposition~\ref{prop:LTS} and why the long-time frequency recursion, finite-mass recovery, and absolute low-frequency estimate are independent of the focusing/defocusing sign. On unions of characteristic intervals, \cite[Theorem~4.1]{KillipVisan2012} establishes the long-time Strichartz recursion; \cite[Lemma~5.1]{KillipVisan2012} combines it with the no-waste Duhamel formula to recover finite mass; and \eqref{eq:absolute-low-mass} is the absolute low-frequency estimate obtained in the proof of \cite[Theorem~5.2]{KillipVisan2012} before conservation of mass is used to exclude the cascade. For $T<T_{\max}$, write $\mathsf K_T:=\mathsf K([0,T])$. The cited results first give a recursion on $[0,T]$ uniform in $T<T_{\max}$ under $\sup_T\mathsf K_T<\infty$, then recover finite mass, and finally yield
\begin{equation}\label{eq:LTS-source-chain}
 \sup_{t<T_{\max}}\|u(t)\|_2<\infty,
 \qquad
 \sup_{t<T_{\max}}\|P_{\le M}u(t)\|_2\lesssim_u M^{1/2}
 \quad(0<M\le1).
\end{equation}

For the focusing adaptation, set
\[
 F_\mu(u):=\mu|u|^4u,
 \qquad \mu\in\{-1,+1\}.
\]
In the long-time recursion, double-Duhamel mass recovery, and low-frequency estimate, the nonlinearity enters only through Strichartz, maximal Strichartz, Bernstein, H\"older, and triangle inequalities. For every relevant norm $X$,
\[
 P_MF_\mu(u)=\mu P_M(|u|^4u),
 \qquad
 \|P_MF_\mu(u)\|_X
 =\|P_M(|u|^4u)\|_X,
\]
and the same is true for differentiated and frequency-decomposed terms. Thus $\mu$ does not enter the recursion constants. The required uniform $\dot H^1$ bound follows from Lemma~\ref{lem:coercivity}. The favorable defocusing interaction-Morawetz sign is used only later in \cite{KillipVisan2012} to exclude the quasi-soliton channel; it is not used in \eqref{eq:LTS-source-chain}. Therefore \eqref{eq:LTS-mass-recovery}--\eqref{eq:absolute-low-mass} hold on the present focusing below-threshold branch.
\end{proof}

\section{The finite-time channel}

We exclude the channel $\mathsf{FT}$. Let $u:[0,T_+)\times\R^3\to\C$ be a one-sided normalized minimal critical element with $\mathfrak T(u)=\mathsf{FT}$. By Theorem~\ref{thm:minimal-element} and Proposition~\ref{prop:one-sided-selection}, Proposition~\ref{prop:no-waste} applies automatically. We first prove a slightly stronger statement whose explicit hypotheses are only the below-threshold bound, a finite endpoint, and the no-waste formula.

A finite-time endpoint first yields decay of the mass in every fixed frequency band.

\begin{lemma}[Frequency-band mass estimate]\label{lem:finite-band}
There exists $C_u<\infty$ such that, for every dyadic $M>0$ and $t<T_+$,
\begin{equation}\label{eq:finite-band-estimate}
 \|P_Mu(t)\|_2\le C_uM(T_+-t).
\end{equation}
\end{lemma}

\begin{proof}
Apply $P_M$ to \eqref{eq:no-waste-finite}. Since the Fourier support of $P_M$ is contained in
\[
 \{\xi:M/4\le|\xi|\le2M\},
\]
the weak $\dot H^1$ and weak $L^2$ topologies are equivalent on this band. Weak lower semicontinuity and unitarity of the Schr\"odinger group on $L^2$ give
\begin{align*}
 \|P_Mu(t)\|_2
 &\le \liminf_{T\uparrow T_+}
 \left\|\int_t^T e^{\ii(t-s)\Delta}P_M(|u|^4u)(s)\dd s\right\|_2\\
 &\le \int_t^{T_+}\|P_M(|u|^4u)(s)\|_2\dd s.
\end{align*}
For every $F\in L^{6/5}(\R^3)$, the three-dimensional Bernstein inequality yields
\[
 \|P_MF\|_2
 \le C M^{3(5/6-1/2)}\|F\|_{6/5}
 =CM\|F\|_{6/5}.
\]
Taking $F=|u|^4u$ and using Sobolev embedding and the below-threshold energy bound, we obtain
\begin{align*}
 \|P_Mu(t)\|_2
 &\le CM\int_t^{T_+}\||u(s)|^4u(s)\|_{6/5}\dd s\\
 &=CM\int_t^{T_+}\|u(s)\|_6^5\dd s\\
 &\le C_uM(T_+-t),
\end{align*}
which is \eqref{eq:finite-band-estimate}.
\end{proof}

Summing the dyadic bands upgrades the preceding estimate to recovery of absolute low-frequency mass.

\begin{lemma}[Low-frequency mass recovery]\label{lem:finite-low}
There exists $C_u<\infty$ such that, for every $M>0$ and $t<T_+$,
\begin{equation}\label{eq:finite-low-estimate}
 \|P_{\le M}u(t)\|_2\le C_uM(T_+-t).
\end{equation}
\end{lemma}

\begin{proof}
Since $P_{\le M}u(t)\in L^2$ is not known a priori, begin with a finite low-frequency truncation. For dyadic $0<\varepsilon<M$, define
\[
 P_{[\varepsilon,M]}:=P_{\le M}-P_{\le\varepsilon}.
\]
Lemma~\ref{lem:finite-band}, through \eqref{eq:finite-band-estimate}, and the Littlewood--Paley square-function estimate give
\begin{align*}
 \|P_{[\varepsilon,M]}u(t)\|_2^2
 &\le C\sum_{\substack{\varepsilon/4\le N\le2M\\N\ \mathrm{dyadic}}}
 \|P_Nu(t)\|_2^2\\
 &\le C_u(T_+-t)^2\sum_{N\le2M}N^2.
\end{align*}
Since $\sum_{N\le2M}N^2\le CM^2$,
\begin{equation}\label{eq:truncated-low-uniform}
 \sup_{0<\varepsilon<M}
 \|P_{[\varepsilon,M]}u(t)\|_2
 \le C_uM(T_+-t).
\end{equation}
On the other hand, $u(t)\in\dot H^1$ implies
\[
 \|P_{\le\varepsilon}u(t)\|_6
 \le C\|\nabla P_{\le\varepsilon}u(t)\|_2
 \longrightarrow0
 \qquad(\varepsilon\downarrow0),
\]
so $P_{[\varepsilon,M]}u(t)\to P_{\le M}u(t)$ in the sense of distributions. The uniform bound \eqref{eq:truncated-low-uniform} gives weak $L^2$ compactness. Every sequence $\varepsilon_n\downarrow0$ therefore has a weakly convergent subsequence, and uniqueness of the distributional limit identifies the weak limit with $P_{\le M}u(t)$. Weak lower semicontinuity yields \eqref{eq:finite-low-estimate}.
\end{proof}

Combining low-frequency recovery with the energy control of the high-frequency tail excludes the finite-time channel.

\begin{proposition}[Exclusion of finite time]\label{prop:finite-time-exclusion}
Let $u:[0,T_+)\times\R^3\to\C$ be a nonzero below-threshold strong solution satisfying Proposition~\ref{prop:no-waste}, with $T_+<\infty$. Then there exists $C_u<\infty$ such that
\begin{equation}\label{eq:finite-time-mass-decay}
 u(t)\in L^2(\R^3),
 \qquad
 \|u(t)\|_2\le C_u(T_+-t)^{1/2}.
\end{equation}
Consequently $u\equiv0$, contradicting nontriviality. Hence the channel $\mathsf{FT}$ is empty.
\end{proposition}

\begin{proof}
For every $M>0$, Plancherel and $|\xi|>M$ give
\begin{equation}\label{eq:high-mass-energy}
 \|P_{>M}u(t)\|_2^2
 \le M^{-2}\|\nabla u(t)\|_2^2
 \le C_uM^{-2}.
\end{equation}
By Lemma~\ref{lem:finite-low}, specifically \eqref{eq:finite-low-estimate}, together with the triangle inequality and \eqref{eq:high-mass-energy},
\begin{equation}\label{eq:mass-two-pieces}
 \|u(t)\|_2
 \le C_u\bigl(M(T_+-t)+M^{-1}\bigr).
\end{equation}
Let $h:=T_+-t>0$ and choose dyadic $M$ such that
\begin{equation}\label{eq:dyadic-optimal-M}
 \frac12h^{-1/2}\le M\le h^{-1/2}.
\end{equation}
Substitution of the dyadic choice \eqref{eq:dyadic-optimal-M} into \eqref{eq:mass-two-pieces} yields
\begin{equation}\label{eq:mass-decay-derived}
 \|u(t)\|_2\le C_uh^{1/2}
 =C_u(T_+-t)^{1/2}.
\end{equation}
Thus $u(t)\in L^2$ and \eqref{eq:finite-time-mass-decay} holds.

Fix $t_0<T_+$. Equation \eqref{eq:mass-decay-derived} gives $u(t_0)\in\dot H^1\cap L^2=H^1$. For every $T\in(t_0,T_+)$, persistence of $H^1$ regularity and uniqueness in the critical class yield
\[
 u\in C_tH_x^1([t_0,T]\times\R^3).
\]
Conservation of mass gives $\|u(T)\|_2=\|u(t_0)\|_2$. Since $T<T_+$ is arbitrary, \eqref{eq:mass-decay-derived} implies
\[
 \|u(t_0)\|_2
 =\lim_{T\uparrow T_+}\|u(T)\|_2
 \le\lim_{T\uparrow T_+}C_u(T_+-T)^{1/2}=0.
\]
Hence $u(t_0)=0$, and uniqueness gives $u\equiv0$, a contradiction.
\end{proof}

\section{The rapid-cascade channel}

We now exclude $\mathsf{RC}$. Fix a one-sided normalized minimal critical element $u$ with $\mathfrak T(u)=\mathsf{RC}$. By Theorem~\ref{thm:minimal-element} and Proposition~\ref{prop:one-sided-selection}, $u$ automatically satisfies Proposition~\ref{prop:no-waste}, and
\begin{equation}\label{eq:N-ge-one}
 N(t)\ge1,
 \qquad
 S_{[0,\infty)}(u)=\infty,
 \qquad t\ge0.
\end{equation}

We first show that finiteness of the scale integral forces the frequency scale to diverge.

\begin{lemma}[Divergence of the scale]\label{lem:N-to-infty}
If
\begin{equation}\label{eq:finite-K-global}
 \int_0^\infty N(t)^{-1}\dd t<\infty,
\end{equation}
then
\begin{equation}\label{eq:N-to-infty}
 \lim_{t\to\infty}N(t)=\infty.
\end{equation}
\end{lemma}

\begin{proof}
Cover $[0,\infty)$ by characteristic intervals $J_k$. By \eqref{eq:characteristic-interval},
\begin{align}
 \int_{J_k}N(t)^{-1}\dd t
 &\ge |J_k|\bigl(\sup_{t\in J_k}N(t)\bigr)^{-1}\notag\\
 &\ge c_uN_k^{-2}(C_uN_k)^{-1}
 =c_u'N_k^{-3},
 \qquad c_u':=c_u/C_u>0.
 \label{eq:K-contribution-lower}
\end{align}
If \eqref{eq:N-to-infty} failed, there would be $A<\infty$, times $t_n\to\infty$, and distinct characteristic intervals $J_{k_n}\ni t_n$ such that $N(t_n)\le A$. By \eqref{eq:characteristic-interval}, $N_{k_n}\le C_uA$. Therefore, using the one-interval lower bound \eqref{eq:K-contribution-lower},
\begin{align*}
 \int_0^\infty N(t)^{-1}\dd t
 &\ge\sum_{n=1}^\infty\int_{J_{k_n}}N(t)^{-1}\dd t\\
 &\ge\sum_{n=1}^\infty c_u'(C_uA)^{-3}=\infty,
\end{align*}
contradicting \eqref{eq:finite-K-global}.
\end{proof}

The scale divergence and low-frequency mass recovery now complete the exclusion.

\begin{proposition}[Exclusion of rapid cascades]\label{prop:rapid-cascade-exclusion}
There is no nonzero forward-global almost-periodic solution satisfying \eqref{eq:N-ge-one}, Proposition~\ref{prop:no-waste}, and \eqref{eq:finite-K-global}. Hence the channel $\mathsf{RC}$ is empty.
\end{proposition}

\begin{proof}
By Proposition~\ref{prop:LTS}, the low-frequency estimate \eqref{eq:absolute-low-mass} holds for some $C_u<\infty$. Lemma~\ref{lem:N-to-infty} also gives the scale divergence \eqref{eq:N-to-infty}, and
\[
 u\in L_t^\infty L_x^2([0,\infty)\times\R^3),
 \qquad
 u(0)\in\dot H^1\cap L^2=H^1.
\]
Persistence of $H^1$ regularity and uniqueness in Proposition~\ref{prop:local-theory}(ii) give
\[
 u\in C_tH_x^1([0,\infty)\times\R^3),
 \qquad
 \M(u(t))=\M(u(0))<\infty
 \quad(t\ge0).
\]
Write this conserved quantity as $\M(u):=\|u(t)\|_2^2$. Enlarging $C_u$ if necessary, assume also $C_u\ge\sup_{t\ge0}\|\nabla u(t)\|_2$.

Let $C_{\mathrm{LP}}\ge1$ depend only on the fixed Littlewood--Paley cutoff and satisfy
\[
 \|(P_{\le L}-P_{\le M})f\|_2
 \le C_{\mathrm{LP}}M^{-1}\|\nabla P_{\le L}f\|_2
 \qquad(L>M>0).
\]
Fix $\varepsilon>0$. First choose $M\in(0,1)$ such that $C_uM^{1/2}\le\varepsilon/4$. Next, by frequency compactness \eqref{eq:ap-frequency} and the support of the smooth cutoff, choose $c\in(0,1)$ so that
\[
 \|\nabla P_{\le cN(t)}u(t)\|_2
 \le \frac{\varepsilon M}{4C_{\mathrm{LP}}}
 \qquad(t\ge0).
\]
Finally, using the scale divergence \eqref{eq:N-to-infty}, choose $t_\varepsilon$ so that, for $t\ge t_\varepsilon$,
\[
 cN(t)>M,
 \qquad
 \frac{C_u}{cN(t)}\le\frac{\varepsilon}{4}.
\]

For such $t$, define
\[
 P_{(M,cN(t)]}:=P_{\le cN(t)}-P_{\le M}.
\]
Then
\begin{equation}\label{eq:three-frequency-decomposition}
 u=P_{\le M}u+P_{(M,cN(t)]}u+P_{>cN(t)}u.
\end{equation}
The three pieces satisfy
\begin{align}
 \|P_{\le M}u(t)\|_2
 &\le C_uM^{1/2}\le\frac{\varepsilon}{4},\notag\\
 \|P_{(M,cN(t)]}u(t)\|_2
 &\le C_{\mathrm{LP}}M^{-1}
 \|\nabla P_{\le cN(t)}u(t)\|_2
 \le\frac{\varepsilon}{4},\notag\\
 \|P_{>cN(t)}u(t)\|_2
 &\le(cN(t))^{-1}\|\nabla u(t)\|_2
 \le\frac{\varepsilon}{4}.
 \label{eq:high-piece}
\end{align}
Hence \eqref{eq:three-frequency-decomposition}--\eqref{eq:high-piece} give
\begin{equation}\label{eq:mass-less-epsilon}
 \|u(t)\|_2\le\varepsilon
 \qquad(t\ge t_\varepsilon).
\end{equation}
The estimate \eqref{eq:mass-less-epsilon}, conservation of mass, and arbitrariness of $\varepsilon$ imply $\M(u)=0$, so $u\equiv0$, a contradiction.
\end{proof}

\section{Logarithmically averaged kernels and conservation laws}

From this point on, $u$ denotes a finite-mass global solution satisfying the hypotheses of Theorem~\ref{thm:finite-mass-rigidity}. Proposition~\ref{prop:local-theory}(ii), conservation of mass, and Lemma~\ref{lem:coercivity} give
\begin{equation}\label{eq:global-H1-budget}
 \M(u(t))=\M(u(0))<\infty,
 \qquad
 \sup_{t\ge0}\|\nabla u(t)\|_2<\infty,
 \qquad
 \|u\|_{L_t^\infty H_x^1([0,\infty))}<\infty.
\end{equation}
In Sections~5--8, we write $\M(u):=\M(u(t))=\M(u(0))$.

\subsection{The logarithmic scale average}

Fix a real-valued radial cutoff $\zeta\in C_c^\infty(\R^3)$ satisfying
\begin{equation}\label{eq:zeta-assumption}
 0\le\zeta\le1,
 \qquad
 \zeta(x)=1\quad(|x|\le1),
 \qquad
 \supp\zeta\subset\{x:|x|\le2\}.
\end{equation}
Define the radial convolution kernel
\begin{equation}\label{eq:Phi-def}
 \Phi(z):=\int_{\R^3}\zeta(w)^2\zeta(w-z)^2\dd w.
\end{equation}
Let $e_1:=(1,0,0)$ and define
\begin{equation}\label{eq:Phi-radial-profile}
 \Phi_{\mathrm{rad}}(r):=\Phi(re_1),
 \qquad r\ge0.
\end{equation}
By \eqref{eq:zeta-assumption},
\begin{equation}\label{eq:Phi-properties}
 \Phi\in C_c^\infty(\R^3),
 \qquad
 \Phi(z)=\Phi_{\mathrm{rad}}(|z|),
 \qquad
 \supp\Phi\subset\{z:|z|\le4\}.
\end{equation}
Given $R_0>0$ and $J\ge1$, define
\begin{align}
 \phi_{J,R_0}(z)
 &:=\frac1J\int_{R_0}^{e^JR_0}
 \Phi(z/R)\frac{\dd R}{R},
 \label{eq:phi-def}\\
 A_{J,R_0}(z)&:=z\phi_{J,R_0}(z).
 \label{eq:A-def}
\end{align}
For fixed $(J,R_0)$, abbreviate $\phi:=\phi_{J,R_0}$ and $A:=A_{J,R_0}$. The measure $\dd R/R$ is uniform in the logarithmic variable $\rho=\log(R/R_0)\in[0,J]$: equal logarithmic scale intervals receive equal weight, while the ratio of the endpoint radii is $e^J$.

We record estimates for the endpoint size, nonflat derivative, and bi-Laplacian of the averaged kernel.

\begin{lemma}[Kernel estimates]\label{lem:kernel-bounds}
For $z\in\R^3$,
\begin{align}
 \partial_kA_j(z)
 &=\delta_{jk}\phi(z)+z_j\partial_k\phi(z),
 \label{eq:A-Hessian}\\
 \operatorname{div}A(z)
 &=3\phi(z)+z\cdot\nabla\phi(z).
 \label{eq:A-divergence}
\end{align}
There exists $C$, depending only on $\zeta$, such that
\begin{align}
 \|A\|_\infty&\le Ce^JR_0,
 \label{eq:A-endpoint-bound}\\
 \|z\otimes\nabla\phi\|_\infty
 +\|z\cdot\nabla\phi\|_\infty
 &\le \frac{C}{J},
 \label{eq:flat-error-bound}\\
 \|\Delta\operatorname{div}A\|_\infty
 &\le \frac{C}{JR_0^2}.
 \label{eq:biharmonic-bound}
\end{align}
\end{lemma}

\begin{proof}
Equations \eqref{eq:A-Hessian}--\eqref{eq:A-divergence} follow by differentiating $A_j(z)=z_j\phi(z)$. If $\Phi(z/R)\ne0$, then \eqref{eq:Phi-properties} gives $|z|\le4R$. Hence
\begin{align*}
 |A(z)|
 &\le \frac{|z|}{J}\int_{R_0}^{e^JR_0}
 |\Phi(z/R)|\frac{\dd R}{R}\\
 &\le \frac{4\|\Phi\|_\infty}{J}\int_{R_0}^{e^JR_0}\dd R
 \le Ce^JR_0,
\end{align*}
proving \eqref{eq:A-endpoint-bound}.

Let $r:=|z|/R$. By \eqref{eq:Phi-radial-profile}, the chain rule gives
\begin{equation}\label{eq:scale-derivative}
 z\cdot\nabla_z\Phi(z/R)
 =r\Phi_{\mathrm{rad}}'(r)
 =-R\partial_R\Phi(z/R).
\end{equation}
Substitution of \eqref{eq:scale-derivative} into \eqref{eq:phi-def} yields
\begin{align}
 z\cdot\nabla\phi(z)
 &=-\frac1J\int_{R_0}^{e^JR_0}\partial_R\Phi(z/R)\dd R\notag\\
 &=\frac1J\left[\Phi(z/R_0)-\Phi(z/(e^JR_0))\right].
 \label{eq:log-telescope}
\end{align}
Since $\phi$ is radial, $z_j\partial_k\phi$ equals
\begin{equation}\label{eq:radial-matrix}
 \frac{z_jz_k}{|z|^2}\,z\cdot\nabla\phi
\end{equation}
for $z\ne0$, with continuous extension at the origin. Equations \eqref{eq:log-telescope}--\eqref{eq:radial-matrix} give \eqref{eq:flat-error-bound}.

For fixed $R$, define
\[
 G_\zeta(y):=\Delta_y\operatorname{div}_y(y\Phi(y)).
\]
Then $G_\zeta\in C_c^\infty(\R^3)$, and scaling gives
\begin{equation}\label{eq:scaled-biharmonic}
 \Delta_z\operatorname{div}_z\bigl(z\Phi(z/R)\bigr)
 =R^{-2}G_\zeta(z/R).
\end{equation}
By \eqref{eq:A-def}, \eqref{eq:scaled-biharmonic}, and boundedness of $G_\zeta$,
\begin{align*}
 \|\Delta\operatorname{div}A\|_\infty
 &\le\frac{C}{J}\int_{R_0}^{e^JR_0}R^{-2}\frac{\dd R}{R}\\
 &=\frac{C}{2J}\left(R_0^{-2}-e^{-2J}R_0^{-2}\right)
 \le\frac{C}{JR_0^2}.
\end{align*}
\end{proof}

\subsection{Local conservation laws}

For a smooth function $v$, define the mass density, momentum density, and stress tensor by
\begin{align}
 m[v](t,x)&:=|v(t,x)|^2,\notag\\
 p_j[v](t,x)&:=\Ima\bigl(\overline{v(t,x)}\,\partial_jv(t,x)\bigr),\notag\\
 T_{jk}[v](t,x)&:=\frac12\delta_{jk}\Delta m[v](t,x)
 -2\Rea\bigl(\partial_jv(t,x)\,\partial_k\overline{v(t,x)}\bigr).
 \label{eq:T-def}
\end{align}
Write $p[v]:=(p_1[v],p_2[v],p_3[v])$ and $T[v]:=(T_{jk}[v])_{1\le j,k\le3}$. For the current solution $u$, abbreviate
\[
 m:=m[u],\qquad p_j:=p_j[u],\qquad p:=p[u],\qquad
 T_{jk}:=T_{jk}[u],\qquad T:=T[u].
\]
Repeated indices $j,k\in\{1,2,3\}$ are summed.

The Morawetz identity begins with the local conservation laws for mass and momentum.

\begin{lemma}[Local conservation laws]\label{lem:local-laws}
If $u$ is a smooth rapidly decaying solution, then
\begin{equation}\label{eq:local-mass-law}
 \partial_tm+2\partial_kp_k=0,
\end{equation}
and
\begin{equation}\label{eq:local-momentum-law}
 \partial_tp_j
 =\partial_kT_{jk}+\frac23\partial_j|u|^6.
\end{equation}
\end{lemma}

\begin{proof}
Equation \eqref{eq:nls} gives
\begin{equation}\label{eq:u-time}
 u_t=\ii(\Delta u+|u|^4u),
 \qquad
 \bar u_t=-\ii(\Delta\bar u+|u|^4\bar u).
\end{equation}
Hence
\begin{align*}
 \partial_tm
 &=u_t\bar u+u\bar u_t\\
 &=\ii\Delta u\,\bar u-\ii u\Delta\bar u\\
 &=-2\partial_k\Ima(\bar u\partial_ku)
 =-2\partial_kp_k,
\end{align*}
which proves \eqref{eq:local-mass-law}.

Differentiating $p_j$ and using \eqref{eq:u-time},
\begin{align}
 \partial_tp_j
 &=\Ima(\bar u_t\partial_ju+\bar u\partial_ju_t)\notag\\
 &=\Rea\left(-\Delta\bar u\,\partial_ju+\bar u\,\partial_j\Delta u\right)\notag\\
 &\quad+\Rea\left(-|u|^4\bar u\,\partial_ju+\bar u\,\partial_j(|u|^4u)\right).
 \label{eq:momentum-expanded}
\end{align}
For the linear term,
\begin{align}
 \frac12\partial_j\Delta|u|^2
 &=\partial_j|\nabla u|^2
 +\Rea\bigl(\partial_j\bar u\,\Delta u+\bar u\,\partial_j\Delta u\bigr),
 \label{eq:half-delta-mass}\\
 \partial_k\Rea(\partial_ju\,\partial_k\bar u)
 &=\Rea(\partial_{jk}u\,\partial_k\bar u)
 +\Rea(\partial_ju\,\Delta\bar u),
 \label{eq:stress-derivative}\\
 \partial_j|\nabla u|^2
 &=2\Rea(\partial_{jk}u\,\partial_k\bar u).
 \label{eq:grad-square-derivative}
\end{align}
Substituting \eqref{eq:stress-derivative}--\eqref{eq:grad-square-derivative} into \eqref{eq:half-delta-mass} yields
\begin{align}
 &\partial_k\left[\frac12\delta_{jk}\Delta|u|^2
 -2\Rea(\partial_ju\,\partial_k\bar u)\right]\notag\\
 &\quad=\Rea\bigl(\bar u\,\partial_j\Delta u-\Delta\bar u\,\partial_ju\bigr).
 \label{eq:linear-momentum-divergence}
\end{align}
For the nonlinear term,
\[
 \partial_j(|u|^4u)
 =|u|^4\partial_ju+4|u|^2\Rea(\bar u\partial_ju)u,
\]
so
\begin{align}
 &\Rea\left(-|u|^4\bar u\,\partial_ju+\bar u\,\partial_j(|u|^4u)\right)\notag\\
 &\quad=4|u|^4\Rea(\bar u\partial_ju)
 =\frac23\partial_j|u|^6.
 \label{eq:nonlinear-momentum-gradient}
\end{align}
Substitution of \eqref{eq:linear-momentum-divergence} and \eqref{eq:nonlinear-momentum-gradient} into \eqref{eq:momentum-expanded} gives \eqref{eq:local-momentum-law}.
\end{proof}

\section{The self-interaction Morawetz identity}

For a smooth finite-mass solution, define
\begin{equation}\label{eq:Morawetz-functional}
 \mathcal I_{J,R_0}(t)
 :=2\iint_{\R^3\times\R^3}
 m(t,y)A_j(x-y)p_j(t,x)\dd x\dd y.
\end{equation}

Substitution of the local conservation laws into the self-interaction functional gives the following exact time-derivative formula.

\begin{proposition}[Morawetz identity]\label{prop:exact-Morawetz}
If $u$ is smooth and rapidly decaying, then
\begin{align}
 \frac{\dd}{\dd t}\mathcal I_{J,R_0}(t)
 ={}&4\iint \partial_kA_j(x-y)
 \Bigl[
 m(y)\Rea(\partial_ju\,\partial_k\bar u)(x)
 -p_j(x)p_k(y)
 \Bigr]\dd x\dd y
 \notag\\
 &-\frac43\iint
 (\operatorname{div}A)(x-y)m(y)|u(x)|^6\dd x\dd y
 \notag\\
 &-\iint
 \Delta(\operatorname{div}A)(x-y)m(x)m(y)\dd x\dd y.
 \label{eq:exact-Morawetz-identity}
\end{align}
Moreover, if $\M(u)<\infty$, then
\begin{equation}\label{eq:Morawetz-endpoint}
 \sup_t|\mathcal I_{J,R_0}(t)|
 \le Ce^JR_0\M(u)^{3/2}
 \sup_t\|\nabla u(t)\|_2.
\end{equation}
\end{proposition}

\begin{proof}
We suppress the time variable. Differentiating \eqref{eq:Morawetz-functional} and using the conservation laws in Lemma~\ref{lem:local-laws}, define
\begin{align*}
 I_m&:=2\iint \partial_tm(y)A_j(x-y)p_j(x)\dd x\dd y,\\
 I_p&:=2\iint m(y)A_j(x-y)\partial_tp_j(x)\dd x\dd y.
\end{align*}
Then
\begin{equation}\label{eq:I-prime-split}
 \mathcal I_{J,R_0}'=I_m+I_p.
\end{equation}
By \eqref{eq:local-mass-law}, integration by parts in $y$, together with
\[
 \partial_{y_k}A_j(x-y)=-\partial_kA_j(x-y),
\]
gives
\begin{align}
 I_m
 &=-4\iint \partial_{y_k}p_k(y)A_j(x-y)p_j(x)\dd x\dd y
 \notag\\
 &=4\iint p_k(y)\partial_{y_k}A_j(x-y)p_j(x)\dd x\dd y
 \notag\\
 &=-4\iint \partial_kA_j(x-y)p_j(x)p_k(y)\dd x\dd y.
 \label{eq:Im-computed}
\end{align}
By \eqref{eq:local-momentum-law},
\begin{align*}
 I_p
 &=2\iint m(y)A_j(x-y)\partial_kT_{jk}(x)\dd x\dd y\\
 &\quad+\frac43\iint m(y)A_j(x-y)\partial_j|u(x)|^6\dd x\dd y.
\end{align*}
Integrating by parts in $x$,
\begin{align}
 I_p
 &=-2\iint m(y)\partial_kA_j(x-y)T_{jk}(x)\dd x\dd y
 \notag\\
 &\quad-\frac43\iint m(y)(\operatorname{div}A)(x-y)|u(x)|^6\dd x\dd y.
 \label{eq:Ip-after-parts}
\end{align}
Using the stress tensor \eqref{eq:T-def},
\begin{align}
 -2\partial_kA_jT_{jk}
 &=-(\operatorname{div}A)\Delta m
 +4\partial_kA_j\Rea(\partial_ju\,\partial_k\bar u).
 \label{eq:stress-expanded}
\end{align}
For the first term in \eqref{eq:stress-expanded}, two further integrations by parts in $x$ yield
\begin{align}
 -\int_{\R^3}(\operatorname{div}A)(x-y)\Delta m(x)\dd x
 &=-\int_{\R^3}\Delta(\operatorname{div}A)(x-y)m(x)\dd x.
 \label{eq:double-parts}
\end{align}
Substituting \eqref{eq:stress-expanded} and \eqref{eq:double-parts} into \eqref{eq:Ip-after-parts} gives
\begin{align}
 I_p
 ={}&4\iint m(y)\partial_kA_j(x-y)
 \Rea(\partial_ju\,\partial_k\bar u)(x)\dd x\dd y
 \notag\\
 &-\frac43\iint
 (\operatorname{div}A)(x-y)m(y)|u(x)|^6\dd x\dd y
 \notag\\
 &-\iint
 \Delta(\operatorname{div}A)(x-y)m(x)m(y)\dd x\dd y.
 \label{eq:Ip-computed}
\end{align}
Substitution of \eqref{eq:Im-computed} and \eqref{eq:Ip-computed} into \eqref{eq:I-prime-split} proves \eqref{eq:exact-Morawetz-identity}.

For the endpoint bound, the pointwise estimate
\[
 |p(t,x)|\le |u(t,x)|\,|\nabla u(t,x)|
\]
and Cauchy--Schwarz give
\begin{equation}\label{eq:p-L1}
 \|p(t)\|_1
 \le\|u(t)\|_2\|\nabla u(t)\|_2
 =\M(u)^{1/2}\|\nabla u(t)\|_2.
\end{equation}
By Lemma~\ref{lem:kernel-bounds}, the endpoint estimate \eqref{eq:A-endpoint-bound} holds. Combining it with \eqref{eq:Morawetz-functional} and \eqref{eq:p-L1},
\begin{align*}
 |\mathcal I_{J,R_0}(t)|
 &\le2\|A\|_\infty
 \left(\int m(y)\dd y\right)
 \left(\int|p(x)|\dd x\right)\\
 &\le Ce^JR_0\M(u)^{3/2}\|\nabla u(t)\|_2.
\end{align*}
Taking the supremum over time proves \eqref{eq:Morawetz-endpoint}.
\end{proof}

\subsection{Local quantities and the main block}

For $R>0$ and $c\in\R^3$, define
\begin{equation}\label{eq:zeta-Rc}
 \zeta_{R,c}(x):=\zeta\left(\frac{x-c}{R}\right).
\end{equation}
Define the localized mass, localized momentum, localized potential-well functional, and coarse localized mass by
\begin{align}
 M_{R,c}(t)
 &:=\int_{\R^3}\zeta_{R,c}(x)^2|u(t,x)|^2\dd x,\notag\\
 P_{R,c}(t)
 &:=\int_{\R^3}\zeta_{R,c}(x)^2p(t,x)\dd x\in\R^3,\notag\\
 K_{R,c}(t)
 &:=\int_{\R^3}\zeta_{R,c}(x)^2
 \bigl(|\nabla u(t,x)|^2-|u(t,x)|^6\bigr)\dd x,\notag\\
 L_{R,c}(t)
 &:=\int_{B(c,2R)}|u(t,x)|^2\dd x.
 \label{eq:local-L}
\end{align}
Clearly,
\begin{equation}\label{eq:M-le-L}
 0\le M_{R,c}(t)\le L_{R,c}(t).
\end{equation}

The following convolution identity rewrites the flat-kernel term in terms of ballwise local quantities.

\begin{lemma}[Convolution formula]\label{lem:convolution}
For all $x,y\in\R^3$ and $R>0$,
\begin{equation}\label{eq:Phi-convolution}
 \Phi((x-y)/R)
 =R^{-3}\int_{\R^3}
 \zeta_{R,c}(x)^2\zeta_{R,c}(y)^2\dd c.
\end{equation}
\end{lemma}

\begin{proof}
In \eqref{eq:Phi-def}, make the change of variables
\[
 c=x-Rw,
 \qquad
 w=\frac{x-c}{R},
 \qquad
 \dd c=R^3\dd w.
\]
Then
\begin{align*}
 R^{-3}\int\zeta_{R,c}(x)^2\zeta_{R,c}(y)^2\dd c
 &=\int\zeta(w)^2
 \zeta\left(w-\frac{x-y}{R}\right)^2\dd w\\
 &=\Phi((x-y)/R).
\end{align*}
\end{proof}

\begin{figure}[htbp]
\centering
\begin{tikzpicture}[scale=0.92, every node/.style={font=\small}]
  \coordinate (X) at (-1.45,0);
  \coordinate (Y) at (1.45,0);
  \begin{scope}
    \clip (X) circle (2.15);
    \fill[gray!25] (Y) circle (2.15);
  \end{scope}
  \draw[thick] (X) circle (2.15);
  \draw[thick] (Y) circle (2.15);
  \fill (X) circle (1.7pt) node[below left=2pt] {$x$};
  \fill (Y) circle (1.7pt) node[below right=2pt] {$y$};
  \coordinate (C) at (0,0.62);
  \fill (C) circle (1.7pt) node[above=3pt] {$c$};
  \draw[dashed] (C)--(X);
  \draw[dashed] (C)--(Y);
  \node at (-2.45,1.75) {$B(x,2R)$};
  \node at (2.45,1.75) {$B(y,2R)$};
\end{tikzpicture}
\caption{Geometry of the local center in the convolution kernel. A center $c$ in the shaded overlap sees both $x$ and $y$.}
\label{fig:convolution-geometry}
\end{figure}
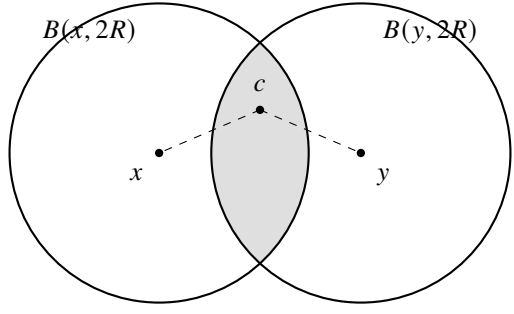

Figure~\ref{fig:convolution-geometry} illustrates the geometry behind \eqref{eq:Phi-convolution}. For fixed $x,y$ and $R$, the two cutoff factors are simultaneously nonzero precisely for centers $c$ in the overlap of the support balls. Thus integration in $c$ reorganizes the interaction of the pair $(x,y)$ into quantities localized at the same center and scale.

Define the flat main block by
\begin{equation}\label{eq:Q-def}
 \mathcal Q_{J,R_0}(t)
 :=\frac4J\int_{R_0}^{e^JR_0}\frac{\dd R}{R}\,R^{-3}
 \int_{\R^3}
 \bigl(M_{R,c}K_{R,c}-|P_{R,c}|^2\bigr)(t)\dd c.
\end{equation}

The convolution formula separates the Morawetz derivative into this main block and controlled remainders.

\begin{proposition}[Main-block decomposition]\label{prop:Q-remainder}
For a smooth rapidly decaying solution, first set $z:=x-y$ and define
\begin{align}
 \mathcal R_{J,R_0}(t)
 :={}&4\iint z_j\partial_k\phi(z)
 \Bigl[
 m(y)\Rea(\partial_ju\,\partial_k\bar u)(x)
 -p_j(x)p_k(y)
 \Bigr]\dd x\dd y
 \notag\\
 &-\frac43\iint
 \bigl(z\cdot\nabla\phi(z)\bigr)m(y)|u(x)|^6\dd x\dd y
 \notag\\
 &-\iint
 \Delta(\operatorname{div}A)(z)m(x)m(y)\dd x\dd y.
 \label{eq:R-explicit}
\end{align}
Then
\begin{equation}\label{eq:I-Q-R}
 \mathcal I_{J,R_0}'(t)
 =\mathcal Q_{J,R_0}(t)+\mathcal R_{J,R_0}(t).
\end{equation}
At every time,
\begin{align}
 |\mathcal R_{J,R_0}(t)|
 &\le \frac{C}{J}\Bigl[
 \M(u(t))\|\nabla u(t)\|_2^2
 +\M(u(t))\|u(t)\|_6^6\Bigr]
 +\frac{C}{JR_0^2}\M(u(t))^2.
 \label{eq:R-bound-explicit}
\end{align}
In particular, under \eqref{eq:global-H1-budget}, there is $C_u<\infty$ such that
\begin{equation}\label{eq:R-bound}
 |\mathcal R_{J,R_0}(t)|
 \le\frac{C_u}{J}+\frac{C_u}{JR_0^2}
 \qquad(t\ge0).
\end{equation}
\end{proposition}

\begin{proof}
Starting from Proposition~\ref{prop:exact-Morawetz}, substitute \eqref{eq:A-Hessian} and \eqref{eq:A-divergence} into \eqref{eq:exact-Morawetz-identity}. The terms containing $z_j\partial_k\phi$, $z\cdot\nabla\phi$, and $\Delta\operatorname{div}A$ constitute \eqref{eq:R-explicit}. The remaining flat part is
\begin{align}
 \mathcal Q_{\mathrm{flat}}
 :={}&4\iint\phi(x-y)
 \Bigl[
 m(y)(|\nabla u(x)|^2-|u(x)|^6)
 -p(x)\cdot p(y)
 \Bigr]\dd x\dd y.
 \label{eq:Q-flat}
\end{align}
By \eqref{eq:phi-def} and Tonelli's theorem,
\begin{align}
 \mathcal Q_{\mathrm{flat}}
 ={}&\frac4J\int_{R_0}^{e^JR_0}\frac{\dd R}{R}
 \iint\Phi((x-y)/R)
 \notag\\
 &\qquad\times
 \Bigl[
 m(y)(|\nabla u(x)|^2-|u(x)|^6)
 -p(x)\cdot p(y)
 \Bigr]\dd x\dd y.
 \label{eq:Q-flat-R}
\end{align}
Applying the convolution formula in Lemma~\ref{lem:convolution} to \eqref{eq:Q-flat-R} gives
\begin{align}
 &\iint\Phi((x-y)/R)
 m(y)(|\nabla u(x)|^2-|u(x)|^6)\dd x\dd y
 =R^{-3}\int_{\R^3}M_{R,c}K_{R,c}\dd c,
 \label{eq:MK-convolution}\\
 &\iint\Phi((x-y)/R)p(x)\cdot p(y)\dd x\dd y
 =R^{-3}\int_{\R^3}|P_{R,c}|^2\dd c.
 \label{eq:PP-convolution}
\end{align}
Equations \eqref{eq:MK-convolution} and \eqref{eq:PP-convolution} show that $\mathcal Q_{\mathrm{flat}}=\mathcal Q_{J,R_0}$, proving \eqref{eq:I-Q-R}.

For the remainder, \eqref{eq:flat-error-bound}, $\int m=\M(u)$, and $\|p\|_1\le\M(u)^{1/2}\|\nabla u\|_2$ imply
\begin{align}
 &\left|4\iint z_j\partial_k\phi(z)
 m(y)\Rea(\partial_ju\,\partial_k\bar u)(x)\dd x\dd y\right|
 \le\frac{C}{J}\M(u)\|\nabla u\|_2^2,
 \label{eq:R-gradient}\\
 &\left|4\iint z_j\partial_k\phi(z)p_j(x)p_k(y)\dd x\dd y\right|
 \le\frac{C}{J}\M(u)\|\nabla u\|_2^2,\notag\\
 &\left|\frac43\iint(z\cdot\nabla\phi(z))m(y)|u(x)|^6\dd x\dd y\right|
 \le\frac{C}{J}\M(u)\|u\|_6^6.\notag
\end{align}
By \eqref{eq:biharmonic-bound},
\begin{equation}\label{eq:R-biharmonic}
 \left|\iint\Delta(\operatorname{div}A)(x-y)m(x)m(y)\dd x\dd y\right|
 \le\frac{C}{JR_0^2}\M(u)^2.
\end{equation}
Combining \eqref{eq:R-gradient}, the two adjacent estimates, and \eqref{eq:R-biharmonic} proves \eqref{eq:R-bound-explicit}; \eqref{eq:global-H1-budget} then gives \eqref{eq:R-bound}.
\end{proof}

\section{Local Galilean coercivity}

\subsection{The optimal local velocity}

When $M_{R,c}(t)>0$, define the local Galilean parameter
\[
 \xi_{R,c}(t):=\frac{P_{R,c}(t)}{M_{R,c}(t)}\in\R^3.
\]
When $M_{R,c}(t)=0$, set $\xi_{R,c}(t)=0$.

The momentum defect in the main block is removed exactly by the optimal ballwise Galilean transformation.

\begin{lemma}[Galilean completion]\label{lem:Galilean-completion}
Let
\[
 v_{R,c}(t,x):=e^{-\ii x\cdot\xi_{R,c}(t)}u(t,x).
\]
Then
\begin{equation}\label{eq:Galilean-identity}
 M_{R,c}K_{R,c}-|P_{R,c}|^2
 =M_{R,c}\int_{\R^3}\zeta_{R,c}^2
 \bigl(|\nabla v_{R,c}|^2-|v_{R,c}|^6\bigr)\dd x.
\end{equation}
\end{lemma}

\begin{proof}
If $M_{R,c}=0$, both sides vanish. Suppose $M_{R,c}>0$. Since
\[
 \nabla v_{R,c}
 =e^{-\ii x\cdot\xi_{R,c}}(\nabla u-\ii\xi_{R,c}u),
\]
we have
\begin{align}
 \int\zeta_{R,c}^2|\nabla v_{R,c}|^2\dd x
 &=\int\zeta_{R,c}^2|\nabla u|^2\dd x
 -2\xi_{R,c}\cdot P_{R,c}
 +|\xi_{R,c}|^2M_{R,c}
 \notag\\
 &=\int\zeta_{R,c}^2|\nabla u|^2\dd x
 -\frac{|P_{R,c}|^2}{M_{R,c}}.
 \label{eq:local-kinetic-completion}
\end{align}
Because $|v_{R,c}|=|u|$, multiplying \eqref{eq:local-kinetic-completion} by $M_{R,c}$ and subtracting
\[
 M_{R,c}\int\zeta_{R,c}^2|u|^6\dd x
\]
proves \eqref{eq:Galilean-identity}.
\end{proof}

Combining this identity with the below-threshold Sobolev gap gives the localized coercive estimate.

\begin{lemma}[Localized coercivity]\label{lem:local-coercivity}
There exist $c_*>0$ and $C<\infty$, where $c_*$ depends only on the below-threshold gap and $C$ only on the fixed cutoff $\zeta$, such that for all $R>0$, $c\in\R^3$, and $t\ge0$,
\begin{align}
 M_{R,c}K_{R,c}-|P_{R,c}|^2
 \ge{}&c_*M_{R,c}
 \|\nabla(\zeta_{R,c}v_{R,c})\|_2^2
 \notag\\
 &-CR^{-2}M_{R,c}L_{R,c}.
 \label{eq:local-coercivity}
\end{align}
\end{lemma}

\begin{proof}
By H\"older's inequality, Sobolev embedding, and \eqref{eq:L6-gap},
\begin{align}
 \int\zeta_{R,c}^2|u|^6\dd x
 &=\int|\zeta_{R,c}v_{R,c}|^2|u|^4\dd x
 \notag\\
 &\le\|\zeta_{R,c}v_{R,c}\|_6^2\|u\|_6^4
 \notag\\
 &\le \cS^2\|u\|_6^4
 \|\nabla(\zeta_{R,c}v_{R,c})\|_2^2
 \notag\\
 &\le(1-\delta_0)
 \|\nabla(\zeta_{R,c}v_{R,c})\|_2^2.
 \label{eq:localized-potential-control}
\end{align}
On the other hand,
\begin{align}
 \|\nabla(\zeta_{R,c}v_{R,c})\|_2^2
 ={}&\int\zeta_{R,c}^2|\nabla v_{R,c}|^2\dd x
 +\int|\nabla\zeta_{R,c}|^2|v_{R,c}|^2\dd x
 \notag\\
 &+2\Rea\int
 \zeta_{R,c}\nabla\zeta_{R,c}\cdot
 \overline{v_{R,c}}\nabla v_{R,c}\dd x.
 \label{eq:cutoff-expand}
\end{align}
Using $2\Rea(\overline v\nabla v)=\nabla|v|^2$ and integrating by parts,
\begin{align}
 2\Rea\int
 \zeta_{R,c}\nabla\zeta_{R,c}\cdot
 \overline{v_{R,c}}\nabla v_{R,c}\dd x
 &=-\int\bigl(\zeta_{R,c}\Delta\zeta_{R,c}
 +|\nabla\zeta_{R,c}|^2\bigr)|v_{R,c}|^2\dd x.
 \label{eq:cross-parts}
\end{align}
Substitution of \eqref{eq:cross-parts} into \eqref{eq:cutoff-expand} gives the exact identity
\begin{equation}\label{eq:cutoff-gradient-identity}
 \int\zeta_{R,c}^2|\nabla v_{R,c}|^2\dd x
 =\|\nabla(\zeta_{R,c}v_{R,c})\|_2^2
 +\int\zeta_{R,c}\Delta\zeta_{R,c}|v_{R,c}|^2\dd x.
\end{equation}
By scaling,
\begin{equation}\label{eq:zeta-laplacian-bound}
 |\zeta_{R,c}\Delta\zeta_{R,c}|
 \le CR^{-2}\one_{B(c,2R)}.
\end{equation}
Combining \eqref{eq:localized-potential-control}, \eqref{eq:cutoff-gradient-identity}, and \eqref{eq:zeta-laplacian-bound} with Lemma~\ref{lem:Galilean-completion}, specifically \eqref{eq:Galilean-identity}, yields
\begin{align*}
 M_{R,c}K_{R,c}-|P_{R,c}|^2
 &\ge \delta_0M_{R,c}
 \|\nabla(\zeta_{R,c}v_{R,c})\|_2^2
 -CR^{-2}M_{R,c}L_{R,c}.
\end{align*}
Taking $c_*:=\delta_0$ proves the lemma.
\end{proof}

\subsection{Nondegeneracy of the core}

The bounded-scale hypothesis converts compactness modulo scaling into compactness modulo translation at a fixed physical scale.

\begin{lemma}[Physical-space precompactness]\label{lem:physical-precompactness}
Suppose that $u$ is almost periodic modulo scaling and translation and that
\[
 0<N_-\le N(t)\le N_+<\infty.
\]
Then
\[
 \mathcal O_u:=\{u(t,\cdot+x(t)):t\ge0\}
\]
is precompact in $\dot H^1(\R^3)$.
\end{lemma}

\begin{proof}
Define
\[
 (D_\lambda f)(x):=\lambda^{1/2}f(\lambda x),
 \qquad
 \overline{\mathcal K_u}^{\,\dot H^1}=: \mathcal K.
\]
Both $\mathcal K$ and $[N_-,N_+]$ are compact, and
\begin{equation}\label{eq:physical-orbit-inclusion}
 u(t,\cdot+x(t))
 =D_{N(t)}\left[N(t)^{-1/2}
 u\left(t,x(t)+\frac{\cdot}{N(t)}\right)\right]
 \in\{D_\lambda f:(f,\lambda)\in\mathcal K\times[N_-,N_+]\}.
\end{equation}
The map $(f,\lambda)\mapsto D_\lambda f$ is jointly continuous from $\dot H^1\times[N_-,N_+]$ to $\dot H^1$. Its image of the compact product set is compact, and \eqref{eq:physical-orbit-inclusion} proves the claim.
\end{proof}

Physical precompactness and nontriviality prevent the concentration core from degenerating.

\begin{lemma}[Core nondegeneracy]\label{lem:core-nondegeneracy}
Under the hypotheses of Theorem~\ref{thm:finite-mass-rigidity}, if $u\not\equiv0$, then there exist $R_*,m_*,\nu_*>0$ such that
\begin{equation}\label{eq:core-mass-l6}
 \int_{B(x(t),R_*)}|u(t,x)|^2\dd x\ge m_*,
 \qquad
 \int_{B(x(t),R_*)}|u(t,x)|^6\dd x\ge\nu_*
\end{equation}
for every $t\ge0$.
\end{lemma}

\begin{proof}
By Lemma~\ref{lem:physical-precompactness} and Corollary~\ref{cor:gradient-nondegenerate},
\begin{equation}\label{eq:orbit-away-zero}
 \mathcal O_u\Subset\dot H^1,
 \qquad
 \inf_{f\in\mathcal O_u}\|\nabla f\|_2>0,
 \qquad
 0\notin\overline{\mathcal O_u}^{\,\dot H^1}.
\end{equation}
Suppose first that no uniform local-mass lower bound exists. Then there are $t_n\ge0$ such that
\[
 \int_{B(x(t_n),n)}|u(t_n,x)|^2\dd x\le n^{-1}.
\]
Set
\begin{equation}\label{eq:f-n-translate}
 f_n(x):=u(t_n,x+x(t_n)).
\end{equation}
By precompactness, after passing to a subsequence,
\begin{equation}\label{eq:f-n-H1-converge}
 f_n\to f\quad\text{in }\dot H^1,
 \qquad
 f_n\to f\quad\text{in }L^6.
\end{equation}
For fixed $R>0$ and $n>R$,
\begin{align*}
 \|f\|_{L^2(B(0,R))}
 &\le\|f-f_n\|_{L^2(B(0,R))}
     +\|f_n\|_{L^2(B(0,R))}\\
 &\le |B(0,R)|^{1/3}\|f-f_n\|_6+n^{-1/2}
 \longrightarrow0.
\end{align*}
Thus $f=0$, contradicting \eqref{eq:orbit-away-zero}--\eqref{eq:f-n-H1-converge}. Hence there exist $R_*^{(2)},m_*>0$ such that
\[
 \int_{B(x(t),R_*^{(2)})}|u(t,x)|^2\dd x\ge m_*
 \qquad(t\ge0).
\]
If no uniform local $L^6$ lower bound exists, the same argument gives times $t_n$ satisfying
\[
 \int_{B(x(t_n),n)}|u(t_n,x)|^6\dd x\le n^{-1}.
\]
Along the subsequence in \eqref{eq:f-n-translate}--\eqref{eq:f-n-H1-converge}, for every fixed $R$,
\[
 \|f\|_{L^6(B(0,R))}
 \le\|f-f_n\|_6+n^{-1/6}
 \longrightarrow0,
\]
again forcing $f=0$. Therefore there exist $R_*^{(6)},\nu_*>0$ with
\[
 \int_{B(x(t),R_*^{(6)})}|u(t,x)|^6\dd x\ge\nu_*
 \qquad(t\ge0).
\]
Taking $R_*:=\max\{R_*^{(2)},R_*^{(6)}\}$ proves \eqref{eq:core-mass-l6}.
\end{proof}

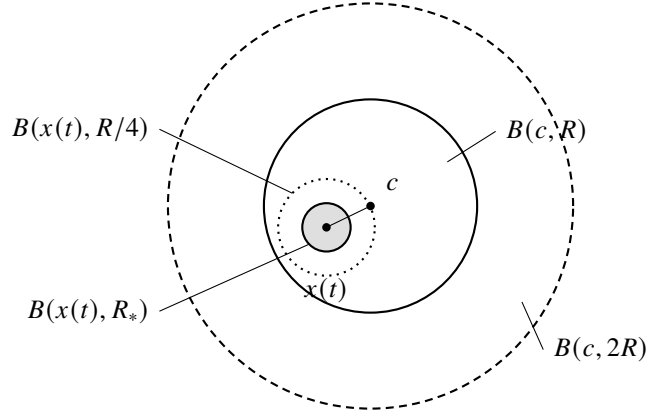
\begin{figure}[htbp]
\centering
\begin{tikzpicture}[scale=0.94, every node/.style={font=\small}]
  \coordinate (Xc) at (0,0);
  \coordinate (Cc) at (0.62,0.30);
  \draw[densely dashed,thick] (Cc) circle (2.85);
  \draw[thick] (Cc) circle (1.50);
  \draw[dotted,thick] (Xc) circle (0.68);
  \fill[gray!25] (Xc) circle (0.34);
  \draw[thick] (Xc) circle (0.34);
  \fill (Xc) circle (1.7pt);
  \node at (0,-0.88) {$x(t)$};
  \fill (Cc) circle (1.7pt) node[above right=2pt] {$c$};
  \draw[thin] (Xc)--(Cc);
  \node[anchor=east] (labA) at (-2.35,1.38) {$B(x(t),R/4)$};
  \draw[thin] (labA.east)--(-0.48,0.48);
  \node[anchor=east] (labB) at (-2.35,-1.18) {$B(x(t),R_*)$};
  \draw[thin] (labB.east)--(-0.25,-0.24);
  \node[anchor=west] (labC) at (2.38,1.35) {$B(c,R)$};
  \draw[thin] (labC.west)--(1.70,0.92);
  \node[anchor=west] (labD) at (3.05,-1.72) {$B(c,2R)$};
  \draw[thin] (labD.west)--(2.75,-1.05);
\end{tikzpicture}
\caption{Containment of the core ball by the localized cutoff ball. The shaded core lies entirely in the region where $\zeta_{R,c}=1$.}
\label{fig:core-inclusion}
\end{figure}

Figure~\ref{fig:core-inclusion} records the geometric containment used below. If $R\ge8R_*$ and $c\in B(x(t),R/4)$, then $B(x(t),R_*)\subset B(c,R)$, so the cutoff equals one on the entire concentration core. Consequently, the local mass and local $L^6$ lower bounds enter the same Galilean block uniformly for a whole family of centers.

This containment yields a positive spatial average of the local Galilean main term.

\begin{lemma}[Positive spatial average]\label{lem:positive-spatial-average}
Under the hypotheses of Theorem~\ref{thm:finite-mass-rigidity}, with $u\not\equiv0$, there exist $R_0<\infty$ and $c_1>0$ such that for all $R\ge R_0$ and $t\ge0$,
\begin{equation}\label{eq:positive-spatial-average}
 R^{-3}\int_{\R^3}M_{R,c}(t)
 \|\nabla(\zeta_{R,c}v_{R,c})(t)\|_2^2\dd c
 \ge c_1.
\end{equation}
\end{lemma}

\begin{proof}
Choose $R_*,m_*,\nu_*$ from Lemma~\ref{lem:core-nondegeneracy} and set
\begin{equation}\label{eq:R0-choice-basic}
 R_0:=8R_*.
\end{equation}
Fix $R\ge R_0$ and $c\in B(x(t),R/4)$. If $x\in B(x(t),R_*)$, then
\begin{equation}\label{eq:x-c-distance}
 |x-c|\le |x-x(t)|+|x(t)-c|
 \le R_*+R/4
 \le3R/8<R.
\end{equation}
By \eqref{eq:zeta-assumption}, \eqref{eq:zeta-Rc}, and \eqref{eq:x-c-distance},
\[
 \zeta_{R,c}(x)=1
 \qquad(x\in B(x(t),R_*)).
\]
Thus \eqref{eq:core-mass-l6} gives
\begin{equation}\label{eq:local-M-lower}
 M_{R,c}(t)\ge m_*,
\end{equation}
and, because $|v_{R,c}|=|u|$,
\[
 \|\zeta_{R,c}v_{R,c}\|_6^6\ge\nu_*.
\]
Sobolev embedding yields
\begin{equation}\label{eq:localized-gradient-lower}
 \|\nabla(\zeta_{R,c}v_{R,c})\|_2^2
 \ge \cS^{-2}\nu_*^{1/3}.
\end{equation}
Integrating \eqref{eq:local-M-lower} and \eqref{eq:localized-gradient-lower} over $c\in B(x(t),R/4)$,
\begin{align*}
 R^{-3}\int_{\R^3}M_{R,c}
 \|\nabla(\zeta_{R,c}v_{R,c})\|_2^2\dd c
 &\ge |B(0,1/4)|m_*\cS^{-2}\nu_*^{1/3}.
\end{align*}
Taking $c_1:=|B(0,1/4)|m_*\cS^{-2}\nu_*^{1/3}$ proves the claim.
\end{proof}

\section{Finite-mass rigidity}

\subsection{A local-mass budget}

The square-integral estimate needed to control the cutoff error is the following Fubini bound.

\begin{lemma}[Fubini bound]\label{lem:Fubini}
For every $R>0$ and $t\ge0$,
\begin{equation}\label{eq:Fubini-bound}
 \int_{\R^3}L_{R,c}(t)^2\dd c
 \le |B(0,2R)|\,\M(u)^2
 \le CR^3\M(u)^2.
\end{equation}
\end{lemma}

\begin{proof}
By \eqref{eq:local-L} and Tonelli's theorem,
\begin{align*}
 \int L_{R,c}^2\dd c
 &=\int\left(\int\one_{B(c,2R)}(x)|u(x)|^2\dd x\right)
 \left(\int\one_{B(c,2R)}(y)|u(y)|^2\dd y\right)\dd c\\
 &=\iint |u(x)|^2|u(y)|^2
 |B(x,2R)\cap B(y,2R)|\dd x\dd y.
\end{align*}
The intersection has volume at most $|B(0,2R)|$. Hence
\[
 \int L_{R,c}^2\dd c
 \le |B(0,2R)|
 \left(\int|u(x)|^2\dd x\right)
 \left(\int|u(y)|^2\dd y\right)
 =|B(0,2R)|\M(u)^2.
\]
\end{proof}

\subsection{Positivity of the main block}

Localized coercivity, the positive spatial average, and the Fubini bound combine to give a uniform lower bound for the logarithmically averaged main block.

\begin{proposition}[Main-block positivity]\label{prop:Q-positive}
Under the hypotheses of Theorem~\ref{thm:finite-mass-rigidity}, with $u\not\equiv0$, let $R_0$ be as in \eqref{eq:R0-choice-basic}. There exist $c_2>0$ and $C_u<\infty$ such that for all $J\ge1$ and $t\ge0$,
\begin{equation}\label{eq:Q-positive-final}
 \mathcal Q_{J,R_0}(t)
 \ge c_2-\frac{C_u}{JR_0^2}.
\end{equation}
\end{proposition}

\begin{proof}
By \eqref{eq:local-coercivity}, for every $R\ge R_0$,
\begin{align}
 &R^{-3}\int_{\R^3}
 \bigl(M_{R,c}K_{R,c}-|P_{R,c}|^2\bigr)\dd c
 \notag\\
 &\quad\ge c_*R^{-3}\int M_{R,c}
 \|\nabla(\zeta_{R,c}v_{R,c})\|_2^2\dd c
 -CR^{-5}\int M_{R,c}L_{R,c}\dd c.
 \label{eq:Q-one-scale-lower}
\end{align}
The first term satisfies
\begin{equation}\label{eq:positive-one-scale}
 c_*R^{-3}\int M_{R,c}
 \|\nabla(\zeta_{R,c}v_{R,c})\|_2^2\dd c
 \ge c_*c_1
\end{equation}
by \eqref{eq:positive-spatial-average}. By \eqref{eq:M-le-L} and the Fubini estimate in Lemma~\ref{lem:Fubini}, namely \eqref{eq:Fubini-bound},
\begin{align}
 R^{-5}\int M_{R,c}L_{R,c}\dd c
 &\le R^{-5}\int L_{R,c}^2\dd c
 \le C\M(u)^2R^{-2}.
 \label{eq:cutoff-error-one-scale}
\end{align}
Combining \eqref{eq:Q-one-scale-lower}, \eqref{eq:positive-one-scale}, and \eqref{eq:cutoff-error-one-scale} gives
\[
 R^{-3}\int
 \bigl(M_{R,c}K_{R,c}-|P_{R,c}|^2\bigr)\dd c
 \ge c_*c_1-C_uR^{-2}.
\]
Substituting into \eqref{eq:Q-def},
\begin{align*}
 \mathcal Q_{J,R_0}(t)
 &\ge\frac4J\int_{R_0}^{e^JR_0}
 \left(c_*c_1-C_uR^{-2}\right)\frac{\dd R}{R}\\
 &\ge4c_*c_1-\frac{C_u}{JR_0^2}.
\end{align*}
Taking $c_2:=4c_*c_1$ proves the proposition.
\end{proof}

\subsection{Passage to low regularity}

The smooth identity and main-block decomposition of Propositions~\ref{prop:exact-Morawetz} and \ref{prop:Q-remainder} must be transferred to the $H^1$ strong solutions allowed by the main theorem.

\begin{proposition}[Low-regularity transfer]\label{prop:low-regularity-transfer}
For a function $v$ for which the relevant integrals are defined, let
\[
 \mathcal I_{J,R_0}[v],\qquad
 \mathcal Q_{J,R_0}[v],\qquad
 \mathcal R_{J,R_0}[v]
\]
denote the functionals obtained by replacing $u$ by $v$ in \eqref{eq:Morawetz-functional}, \eqref{eq:Q-def}, and \eqref{eq:R-explicit}, respectively. Brackets are omitted for the current solution $u$.

Let $u\in L_t^\infty H_x^1([0,T]\times\R^3)$ be a strong solution. For fixed $J\ge1$ and $R_0>0$,
\begin{equation}\label{eq:integrated-Morawetz}
 \mathcal I_{J,R_0}(t_2)-\mathcal I_{J,R_0}(t_1)
 =\int_{t_1}^{t_2}
 \bigl(\mathcal Q_{J,R_0}(t)+\mathcal R_{J,R_0}(t)\bigr)\dd t
\end{equation}
holds for all $0\le t_1<t_2\le T$. The endpoint estimate \eqref{eq:Morawetz-endpoint} and the quantitative remainder estimate \eqref{eq:R-bound-explicit} also hold for $u$. In particular, there is a constant $C_{u,T}$ depending only on $\|u\|_{L_t^\infty H_x^1([0,T])}$ such that
\begin{equation}\label{eq:R-bound-finite-interval}
 |\mathcal R_{J,R_0}(t)|
 \le\frac{C_{u,T}}J+\frac{C_{u,T}}{JR_0^2}
 \qquad(0\le t\le T).
\end{equation}
\end{proposition}

\begin{proof}
Choose $u_{0,n}\in\Sspace(\R^3)$ with
\begin{equation}\label{eq:smooth-data-approx}
 u_{0,n}\to u(0)\quad\text{in }H^1.
\end{equation}
Since $u$ is a strong solution and $[0,T]$ is compact, $S_{[0,T]}(u)<\infty$. Choose $\eta_*>0$ sufficiently small for the $H^1$ stability statement in Proposition~\ref{prop:local-theory}(iii) to be iterated, and then choose a finite partition
\[
 0=\tau_0<\tau_1<\cdots<\tau_L=T,
 \qquad
 S_{[\tau_{\ell-1},\tau_\ell]}(u)\le\eta_*
 \quad(1\le\ell\le L).
\]
Let $u_n$ be the smooth maximal-lifespan solution with initial data $u_{0,n}$. Stability and \eqref{eq:smooth-data-approx} imply that, for all sufficiently large $n$, $u_n$ exists on $[0,T]$ and
\begin{equation}\label{eq:H1-solution-convergence}
 u_n\to u
 \quad\text{in }C_tH_x^1([0,T]\times\R^3)
 \cap L^{10}_{t,x}([0,T]\times\R^3),
\end{equation}
with
\begin{equation}\label{eq:uniform-H1-approximants}
 \sup_n\|u_n\|_{L_t^\infty H_x^1([0,T])}<\infty.
\end{equation}
Uniformly for $t\in[0,T]$,
\begin{align}
 \bigl\||u_n|^2-|u|^2\bigr\|_1
 &\le(\|u_n\|_2+\|u\|_2)\|u_n-u\|_2\to0,
 \label{eq:mass-density-convergence}\\
 \|p[u_n]-p[u]\|_1
 &\le\|u_n-u\|_2\|\nabla u_n\|_2
   +\|u\|_2\|\nabla(u_n-u)\|_2\to0,\notag\\
 \bigl\||u_n|^6-|u|^6\bigr\|_1
 &\le C(\|u_n\|_6^5+\|u\|_6^5)\|u_n-u\|_6\to0,\notag\\
 \|\partial_ju_n\,\partial_k\overline{u_n}
   -\partial_ju\,\partial_k\bar u\|_1
 &\le\|\nabla(u_n-u)\|_2
       (\|\nabla u_n\|_2+\|\nabla u\|_2)\to0.
 \label{eq:stress-density-convergence}
\end{align}
For fixed $J,R_0$,
\begin{equation}\label{eq:fixed-kernel-bounded}
 A,\quad\nabla A,\quad\Delta\operatorname{div}A\in L^\infty(\R^3).
\end{equation}
If $K\in L^\infty$ and $f_n\to f$, $g_n\to g$ in $L^1$, then
\begin{align}
 &\left|\iint K(x-y)\bigl(f_n(x)g_n(y)-f(x)g(y)\bigr)\dd x\dd y\right|
 \notag\\
 &\quad\le\|K\|_\infty
 \left(\|f_n-f\|_1\|g_n\|_1
       +\|f\|_1\|g_n-g\|_1\right)
 \longrightarrow0.
 \label{eq:bounded-kernel-product-convergence}
\end{align}
For $\mathcal Q_{J,R_0}$, use the flat double-integral representation \eqref{eq:Q-flat}, rather than passing to the limit through the local representation involving $P_{R,c}/M_{R,c}$. The solution convergence \eqref{eq:H1-solution-convergence}, together with \eqref{eq:mass-density-convergence}--\eqref{eq:stress-density-convergence} and \eqref{eq:bounded-kernel-product-convergence}, yields convergence of the individual terms in $\mathcal I$, $\mathcal Q$, and $\mathcal R$, uniformly in time. The bounds \eqref{eq:uniform-H1-approximants} and \eqref{eq:fixed-kernel-bounded} give an integrable majorant. Passing to the limit in the smooth identity for $u_n$ proves \eqref{eq:integrated-Morawetz}. The endpoint and remainder estimates pass to the limit by the same argument, and taking the $H^1$ supremum on $[0,T]$ gives \eqref{eq:R-bound-finite-interval}.
\end{proof}

Under the additional assumptions of the main theorem, the preceding coercive estimates remain valid at $H^1$ regularity.

\begin{corollary}[Transfer of coercivity]\label{cor:coercivity-low-regularity}
Let $u$ be an $H^1$ strong solution. If $u$ satisfies the below-threshold condition \eqref{eq:main-threshold}, then \eqref{eq:local-coercivity} holds at every time. If, in addition, $u\not\equiv0$ and satisfies the almost-periodicity and bounded-scale hypotheses of Theorem~\ref{thm:finite-mass-rigidity}, then \eqref{eq:Q-positive-final} holds uniformly in time.
\end{corollary}

\begin{proof}
At a fixed time, the $H^1$ condition and \eqref{eq:L6-gap} give
\[
 u(t)\in H^1,
 \qquad
 \cS^2\|u(t)\|_6^4\le1-\delta_0.
\]
The proof of Lemma~\ref{lem:local-coercivity} therefore applies without any smoothness assumption. Under almost periodicity and \eqref{eq:main-bounded-scale}, Lemmas~\ref{lem:core-nondegeneracy} and \ref{lem:positive-spatial-average}, followed by Proposition~\ref{prop:Q-positive}, give the uniform main-block lower bound.
\end{proof}

The final absorption uses three scales. Once $R_0$ is fixed, the flat main block has a strictly positive time-independent lower bound, the nonflat kernel errors are $O(J^{-1})+O((JR_0^2)^{-1})$, and the Morawetz endpoint is bounded for fixed $J$. Thus one first fixes $R_0$, then chooses $J_0$ to absorb the errors, and only afterward lets the time length tend to infinity. No parameter depends on a later choice.

\subsection{The final contradiction}

\begin{proof}[Proof of Theorem~\ref{thm:finite-mass-rigidity}]
The result is immediate if $u\equiv0$. Suppose instead that $u\not\equiv0$. By \eqref{eq:Q-positive-final} and \eqref{eq:R-bound},
\[
 \mathcal Q_{J,R_0}(t)+\mathcal R_{J,R_0}(t)
 \ge c_2-\frac{C_u}{J}-\frac{C_u}{JR_0^2}.
\]
The radius $R_0$ has already been fixed by core nondegeneracy and is independent of $J$ and of the terminal time. Choose $J_0\ge1$ such that
\[
 \frac{C_u}{J_0}+\frac{C_u}{J_0R_0^2}\le\frac{c_2}{2}.
\]
Then
\[
 \mathcal Q_{J_0,R_0}(t)+\mathcal R_{J_0,R_0}(t)
 \ge\frac{c_2}{2}
 \qquad(t\ge0).
\]
By Proposition~\ref{prop:low-regularity-transfer}, specifically \eqref{eq:integrated-Morawetz}, and Corollary~\ref{cor:coercivity-low-regularity}, for every $T>0$,
\begin{equation}\label{eq:linear-growth}
 \mathcal I_{J_0,R_0}(T)-\mathcal I_{J_0,R_0}(0)
 \ge\frac{c_2}{2}T.
\end{equation}
On the other hand, \eqref{eq:Morawetz-endpoint} yields the finite constant
\[
 B_{u,J_0,R_0}
 :=Ce^{J_0}R_0\M(u)^{3/2}
 \sup_{t\ge0}\|\nabla u(t)\|_2
\]
and hence
\begin{equation}\label{eq:endpoint-two}
 |\mathcal I_{J_0,R_0}(T)-\mathcal I_{J_0,R_0}(0)|
 \le2B_{u,J_0,R_0}.
\end{equation}
Equations \eqref{eq:linear-growth} and \eqref{eq:endpoint-two} imply
\[
 \frac{c_2}{2}T\le2B_{u,J_0,R_0}
 \qquad\text{for every }T>0,
\]
a contradiction as $T\to\infty$.
\end{proof}

\section{The four-channel reduction}

\begin{proof}[Proof of Theorem~\ref{thm:closed-channels}]
Let $u_c$ be a nonzero one-sided normalized minimal critical element. By Theorem~\ref{thm:minimal-element} and Proposition~\ref{prop:one-sided-selection}, the solution automatically satisfies Proposition~\ref{prop:no-waste}; thus the no-waste Duhamel formula used below is not an additional hypothesis.

If $\mathfrak T(u_c)=\mathsf{FT}$, then $T_+<\infty$, and Proposition~\ref{prop:finite-time-exclusion} gives $u_c\equiv0$, contradicting nontriviality.

If $\mathfrak T(u_c)=\mathsf{RC}$, then $T_+=\infty$, $N(t)\ge1$, and $\mathsf K_+(u_c)<\infty$. Proposition~\ref{prop:rapid-cascade-exclusion} again gives $u_c\equiv0$, a contradiction.

Finally suppose that $\mathfrak T(u_c)=\mathsf{FM}$. By \eqref{eq:intro-channel-FM}, there is $t_0\ge0$ such that
\[
 u_c(t_0)\in\dot H^1\cap L^2=H^1,
 \qquad
 1\le N(t)\le N_+<\infty
 \quad(t\ge t_0).
\]
Persistence of $H^1$ regularity and uniqueness in the $\dot H^1$ class, supplied by Proposition~\ref{prop:local-theory}(ii), give
\[
 u_c\in C_tH_x^1([t_0,\infty)\times\R^3).
\]
After translating time, the tail solution still satisfies the below-threshold condition, almost periodicity, and the bounded-scale hypothesis. Theorem~\ref{thm:finite-mass-rigidity} therefore gives $u_c(t)=0$ for $t\ge t_0$. Uniqueness then yields $u_c\equiv0$, again a contradiction. Hence
\[
 \mathfrak T(u_c)\notin\{\mathsf{FT},\mathsf{RC},\mathsf{FM}\}.
\]
\end{proof}

\begin{proof}[Proof of Corollary~\ref{cor:residual-channel}]
If (a) holds, there is no nonscattering minimal critical element, and in particular no one-sided normalized minimal element with $\mathfrak T(u_c)=\mathsf{RQ}$. Thus (b) holds.

Conversely, suppose that (a) fails. After applying the time-reversal symmetry
\[
 u(t,x)\longmapsto\overline{u(-t,x)}
\]
if necessary, we may assume that scattering fails forward in time. Theorem~\ref{thm:minimal-element} and Proposition~\ref{prop:one-sided-selection} produce a nonzero one-sided normalized minimal critical element $u_c$. The four-channel partition gives
\[
 \mathfrak T(u_c)\in\{\mathsf{FT},\mathsf{RC},\mathsf{FM},\mathsf{RQ}\}.
\]
Theorem~\ref{thm:closed-channels} excludes the first three channels, and hence $\mathfrak T(u_c)=\mathsf{RQ}$, contradicting (b). Therefore (a) holds.
\end{proof}

\section*{Acknowledgements}
The authors would like to thank the supporting agencies for their continued support.

\section*{Funding}
The first author was supported by the Project ``Research on Nonlinear Partial Differential Equations'' (No. 2024KYCXTD018), the Special Projects in Key Areas of Guangdong Province (No. ZDZX1088), and the Fund of Guangzhou Municipal Science and Technology (No. 202102080428).

\section*{Disclosure statement}
No potential conflict of interest was reported by the authors.

\section*{Declaration of generative AI use}
The authors report generative AI was not used in their research or preparation of this manuscript.

\section*{Data availability statement}
No data were used in this study.

\end{document}